\documentclass[12pt,usenames,dvipsnames]{amsart}
\usepackage{mathpazo}
\usepackage{hyperref}
\usepackage{cleveref}
\usepackage{bm}
\usepackage{comment}
\usepackage[margin=1in]{geometry}
\usepackage{enumitem}
\usepackage{calrsfs}
\DeclareMathAlphabet{\pazocal}{OMS}{zplm}{m}{n}

\usepackage{graphicx}
\usepackage{tikz-cd}

\usepackage[cmtip, all]{xy}
\usepackage{array}

\usepackage{amssymb,amsmath,latexsym,amsthm}

\newcommand{\bZ}{\mathbb{Z}}
\newcommand{\pO}{\pazocal{O}}
\DeclareMathOperator{\Hom}{Hom}

\DeclareMathOperator{\SL}{SL}

\DeclareMathOperator{\GSp}{GSp}
\DeclareMathOperator{\GL}{GL}

\DeclareMathOperator{\Spec}{Spec}

\DeclareMathOperator{\Aut}{Aut}
\DeclareMathOperator{\Jac}{Jac}
\DeclareMathOperator{\Sp}{Sp}

\DeclareMathOperator{\val}{val}

\DeclareMathOperator{\Span}{span}
\DeclareMathOperator{\Out}{Out}
\DeclareMathOperator{\id}{id}
\DeclareMathOperator{\image}{im}

\DeclareMathOperator{\alg}{alg}

\DeclareMathOperator{\et}{\acute{e}t}

\newcommand{\nhatell}{\hat{n}^{(\ell)}}
\newcommand{\cB}{\mathcal{B}}
\newcommand{\pB}{\pazocal{B}}
\newcommand{\cG}{\mathcal{G}}
\newcommand{\cM}{\mathcal{M}}
\newcommand{\field}[1]{\mathbb{#1}}

\newcommand{\Q}{\field{Q}}

\newcommand{\Zl}{\field{Z}_\ell}
\newcommand{\Z}{\field{Z}}

\newcommand{\Kbar}{\overline{K}}

\newcommand{\R}{\field{R}}
\newcommand{\C}{\field{C}}

\newcommand{\ra}{\rightarrow}

\newcommand{\bs}{\backslash}

\newcommand{\cI}{\mathcal{I}}

\newcommand{\aab}[3]{\alpha_{#1} \wedge \alpha_{#2} \wedge \beta_{#3} }
\newcommand{\abb}[3]{\alpha_{#1} \wedge \beta_{#2} \wedge \beta_{#3} }
\newcommand{\bbb}[3]{\beta_{#1} \wedge \beta_{#2} \wedge \beta_{#3} }

\newcommand{\LT}{\mathbf{LT}}
\newcommand{\pL}{\pazocal{L}}
\newcommand{\pH}{\pazocal{H}}
\newcommand{\pF}{\pazocal{F}}

\newcommand{\LH}{L/H}

\newcommand{\CCt}{\C(\!(t)\!)}

\newcommand{\cD}{\mathcal{D}}
\newcommand{\cK}{\mathcal{K}}
\newcommand{\cX}{\mathcal{X}}

\newcommand{\fX}{\mathfrak{X}}

\newcommand{\cY}{\mathcal{Y}}

\newcommand{\fY}{\mathfrak{Y}}

\newcommand{\bG}{\mathbf{G}}

\DeclareMathOperator{\rank}{rank}

\DeclareMathOperator{\gr}{gr}
\DeclareMathOperator{\coker}{coker}

\newtheorem{theorem}{Theorem}[section]
\newtheorem{lemma}[theorem]{Lemma}
\newtheorem{proposition}[theorem]{Proposition}
\newtheorem{corollary}[theorem]{Corollary}
\newtheorem{example}[theorem]{Example}
\newtheorem*{theorem*}{Theorem}

\theoremstyle{definition}

\theoremstyle{remark}
\newtheorem{remark}[theorem]{Remark}

\newtheorem{definition}[theorem]{Definition}  

\numberwithin{equation}{section}
\numberwithin{table}{section}
\numberwithin{figure}{section}

\newcommand\blfootnote[1]{%
  \begingroup
  \renewcommand\thefootnote{}\footnote{#1}%
  \addtocounter{footnote}{-1}%
  \endgroup
}

\newcommand\classification[2][]{%
  \gdef\@classification{%
    \href{http://www.ams.org/msc/}%
{\textit{2020 Mathematics Subject Classification}} \ignorespaces#2\unskip}}

\title{The Ceresa class:  tropical, topological, and algebraic}

\author{Daniel Corey}
\email{corey@math.tu-berlin.de}
\address{Technische Universit\"at Berlin, Stra{\ss}e des 17 Juni 135, Berlin 10623, Germany.}

\author{Jordan Ellenberg}
\email{ellenber@math.wisc.edu}
\address{UW-Madison, Van Vleck Hall, 480 Lincoln Dr. Madison, WI 53706.}

\author{Wanlin Li}
\email{wanlinli@mit.edu}
\address{MIT, 77 Massachusetts Avenue, Cambridge, MA 02139-4307}

\classification{14T25 (primary), and 14H30, 15C50, 57K20 (secondary)}

\begin{document}
	
\maketitle

\begin{abstract}
 The {\em Ceresa cycle} is an algebraic cycle attached to a smooth algebraic curve with a marked point, which is trivial when the curve is hyperelliptic with a marked Weierstrass point. The image of the Ceresa cycle under a certain cycle class map provides a class in \'etale cohomology called the {\em Ceresa class}. Describing the Ceresa class explicitly for non-hyperelliptic curves is in general not easy.  We present a ``combinatorialization'' of this problem, explaining how to define a Ceresa class for a tropical algebraic curve, and also for a topological surface endowed with a multiset of commuting Dehn twists (where it is related to the Morita cocycle on the mapping class group). We explain how these are related to the Ceresa class of a smooth algebraic curve over $\mathbb{C}(\!(t)\!)$, and show that the Ceresa class in each of these settings is torsion. 
	
\end{abstract}

\section{Introduction}

\blfootnote{\textit{2020 Mathematics Subject Classification}: 14T25 (primary), and 14H30, 15C50, 57K20 (secondary) \\ 
\textit{Keywords}: algebraic cycles, hyperelliptic curves, mapping class group, monodromy,  tropical curves}
When $X$ is a smooth algebraic curve with a marked point over a field, there is a canonical algebraic $1$-cycle on the Jacobian of $X$ called the {\em Ceresa cycle}.  The Ceresa cycle is homologically trivial, but, as Ceresa showed in \cite{Ceresa}, it is not algebraically equivalent to zero for a very general curve of genus greater than $2$.  In some sense it is the simplest non-trivial canonical algebraic cycle ``beyond homology" and as such it has found itself relevant to many natural problems in the geometry of curves and their Jacobians \cite{hainmatsumoto, totaro:chow16, Zhang}.

In recent years, many useful notions in algebraic geometry, and especially in the geometry of algebraic curves, have been seen to carry over to the tropical context, where they become interesting combinatorial notions.  The motivation for the present paper is to understand whether the theory of the Ceresa cycle (or, more precisely, a cohomology class associated to that cycle) can be given a meaningful interpretation in the tropical setting.  In particular, since a tropical curve is just a graph with positive real lengths assigned to the edges and integer weights assigned to the vertices, the Ceresa cycle would be a combinatorial invariant of such a graph.  We define such an invariant in the present paper and begin to investigate its properties.  We show, for example, that the Ceresa class of any hyperelliptic graph is zero (in conformity with the classical case) but that the Ceresa class of the complete graph on four vertices with all edges of length $1$ is a nonzero class of order 16, see Proposition \ref{prop:hyperellipticCeresaTrivial} and Remark \ref{rmk:K4All1}, respectively.  Moreover, we show in Example~\ref{ex:K4} that the Ceresa class is nonzero for {\em every} tropical curve whose underlying graph is the complete graph on four vertices.

Our approach is to model a tropical curve with integral edge lengths as the tropicalization of a curve that degenerates to a stable curve.  We start by considering an algebraic family of smooth complex curves of genus $g$ over a punctured disc $D \bs *$, which degenerates to a stable curve over the central fiber $*$.  There are several lenses through which one can view such a degeneration.

\begin{itemize}
    \item[-] {\em Topology}: The family of complex genus-$g$ curves over $D \bs *$, considered as a manifold, is homotopic to a family of genus-$g$ surfaces fibered over the circle, which we can think of as obtained by taking $\Sigma_g \times [0,1]$ and identifying $\Sigma_g \times 0$ with $\Sigma_g \times 1$ via a diffeomorphism of $\Sigma_g$ defined up to homotopy, i.e., an element of the mapping class group.  The stable reduction implies that this mapping class is a {\em multitwist}; that is, product of integral powers of commuting Dehn twists.  Which twists they are can be read off the dual graph of the stable fiber at $*$, and which powers of each twist appear is determined by the multiplicity of the nodes in the degeneration.
    \item[-] {\em Tropical geometry}:  It is well known that a stable degeneration gives rise to a {\em tropical curve}, which is to say a vertex weighted metric graph; in this case, it will be the dual graph of the stable fiber, with edge-lengths determined by the multiplicity with which the family of curves strikes various boundary components of $\overline{\cM}_g$.
    \item[-] {\em Algebraic geometry over a local field}: We can also restrict the holomorphic family to an infinitesimal neighborhood of $*$, yielding an algebraic curve over $\CCt$ with stable reduction.
\end{itemize}

In each case, there is a certain combinatorial datum which describes the degeneration: in the first case, the mapping class; in the second case, the tropical curve itself; and in the third case the action of the (pro-cyclic) absolute Galois group of $\CCt$ on the \'{e}tale fundamental group of $X_{\overline{\CCt}}$ (or, as we shall see, just on the quotient of that fundamental group by the third term of its lower central series.)  These three data agree in a sense made precise in \S \S \ref{sec:connectionToArithmeticGeometry},  \ref{sec:connectionToTropicalGeometry}. 

The only one of these contexts in which there is a literal Ceresa class is the third one.  But we shall see that we can in fact define the Ceresa cycle directly from the combinatorial datum.  Thus we may now speak of the Ceresa class of a multitwist in the mapping class group, or the Ceresa class of a unipotent automorphism of the geometric \'{e}tale fundamental group of a curve over $\CCt$, or the Ceresa class of a tropical curve with integral edge lengths.  (In this last case, our definition should be compared with that proposed by Zharkov in \cite{zharkov}; see Remark~\ref{rem:zharkov} for some speculations about this.) We explain in \S\ref{subsec:tropicalCeresa} how to extend the definition to non-integral edge lengths.

The topological definition rests fundamentally on the Morita cocycle on the mapping class group~\cite{Morita} (an extension of the Johnson homomorphism).  For the algebraic story  we use in a crucial way the work of Hain and Matsumoto~\cite{hainmatsumoto} relating the Ceresa class in \'{e}tale cohomology (over any field $K$) to the Galois action on the $2$-nilpotent fundamental group.  Indeed, we could just as well have described this paper as being about the ``Morita class" rather than the ``Ceresa class" --- it is the $\ell$-adic Harris-Pulte theorem of Hain and Matsumoto~\cite[\S 8]{hainmatsumoto} that relates the Morita class in group cohomology with the image of the Ceresa cycle under the cycle class map.  

In fact, most of the proofs and theorems in the paper are carried out not in an algebro-geometric context but in the setting of the mapping class group, which seems to be the easiest to work with in practice.  In this context, the Ceresa class of a multitwist can be described quite simply.  Recall that an element of the mapping class group $\Gamma_g$ is {\em hyperelliptic} if it commutes with some hyperelliptic involution $\tau \in \Gamma_g$.  The Ceresa class of a multitwist $\gamma \in \Gamma_g$ is an obstruction to the existence of an element $\tilde{\tau}$ of $\Gamma_g$ which acts as $-I$ on homology (a hyperelliptic involution being an example of such a $\tilde{\tau}$) such that the commutator $[\tilde{\tau},\gamma]$ lies in the Johnson kernel.  In shorter terms, one might say a multitwist has Ceresa class zero if it is ``hyperelliptic up to the Johnson kernel." 

Our main theorem, Theorem~\ref{thm:nonLagrangianTorsion}, is that the Ceresa class we define is torsion for any multitwist (and thus for any tropical curve with integral edge lengths). The Ceresa cycle of a very general complex algebraic curve is known to be non-torsion modulo algebraic equivalence \cite[Theorem~3.2.1]{Top}. So in some sense, our theorem shows that the \'etale Ceresa class defined here is throwing away a lot of information about the algebraic cycle; this is not surprising, since as we shall see it is determined by purely numerical data about the degeneration of a curve in a $1$-dimensional family.  On the other hand, the Ceresa class is readily computable and implies nontriviality of the Ceresa cycle if it is nonzero.  One might make the following analogy; if $K$ is a discrete valuation ring and $P$ is a point on an elliptic curve $E/K$ with bad reduction, then knowledge of the Ceresa cycle is akin to identifying $P$, while knowledge of the Ceresa class is more like knowing which component of the N\'{e}ron fiber $P$ reduces to.   

The fact that our Ceresa class is readily computable is significant because there are few examples of specific curves where the Ceresa cycle or \'etale Ceresa class is known to be trivial or non-trivial. One such example is the Fermat quartic curve, whose Ceresa cycle was found to be not algebraically equivalent to 0 in \cite{HarrisFermat} using the construction of harmonic volume in \cite{Harris}. The \'etale Ceresa class was computed and determined to be nontrivial (in fact, non-torsion) for some examples in \cite[\S 3.4]{Top}. The result \cite[Theorem 1.1]{BLLS} exhibits an explicit curve over a number field whose \'etale Ceresa class is torsion, but does not determine its exact order or prove it to be nontrivial. It is an open problem whether there exist positive dimensional families of non-hyperelliptic curves with torsion Ceresa cycle modulo algebraic equivalence \cite[p.28]{Top}.

The paper is structured as follows.  In \S2 we define the Ceresa class of a multitwist.  In \S3 we explain the relation between the topological definition and the \'etale Ceresa class in algebraic geometry, and in \S4 we explain how the definition extends to a tropical curve with arbitrary real edge lengths.  In \S\S5-6 we prove Theorem~\ref{thm:nonLagrangianTorsion} and describe a finite group in which the tropical Ceresa class naturally lies, a group which might be thought of as a sort of tropical intermediate Jacobian.  Finally, in \S7, we compute the Ceresa classes of several low-genus graphs.  We close with a question: are there non-hyperelliptic tropical curves with Ceresa class zero?

\section{The topological Ceresa class}
	\label{sec:topologicalCeresa}
	
\subsection{The Mapping Class Group and the Symplectic Representation}
\label{subsec:mappingClassGroup}
In this subsection, we recall some basic facts about the mapping class group, see \cite{FarbMargalit} for a detailed treatment. 

Throughout the paper, let $\Sigma_g$ denote a closed genus $g$ surface and $\Sigma_g^1$  denote a genus $g$ surface with one puncture.  Let $\Gamma_g$ (resp. $\Gamma_g^1$) be the mapping class group of $\Sigma_g$ (resp. $\Sigma_g^1$), i.e., the group of isotopy classes of orientation-preserving diffeomorphisms of the surface, and $\Pi_g = \pi_1(\Sigma_g)$.  These groups fit into the Birman exact sequence: 
\begin{equation}
\label{eq:Birman}
1 \to \Pi_g \to \Gamma_g^1 \to \Gamma_g \to 1.
\end{equation}
Given a simple closed curve $a$ in $\Sigma_{g}$ (resp. $\Sigma_g^1$), denote by $T_a$ the left-handed \textit{Dehn twist} of $a$. A \textit{separating twist} is a Dehn twist $T_a$ where $a$ is a separating curve, and a \textit{bounding pair map} is $T_{a}T_{b}^{-1}$ where $a$ and $b$ are homologous non-separating, non-intersecting, simple closed curves.

The singular homology group $H_1(\Sigma_g^1,\mathbb{Z}) \cong H_1(\Sigma_g,\mathbb{Z})$, which we denote by $H$, has a symplectic structure given by the the algebraic intersection pairing $\hat{i}: H \wedge H \to \mathbb{Z}$. The action of $\Gamma_g$ (resp. $\Gamma_g^1$) on $H$ respects this pairing. This yields the symplectic representation of $\Gamma_g$ (resp. $\Gamma_g^1$), and we have the short exact sequence
\begin{equation}
\label{eq:symplecticRepresentation}
    1 \to \cI_g \text{ (resp. }\cI_g^1 \text{)} \to \Gamma_g \text{ (resp. }\Gamma_g^1 \text{)} \to \Sp(2g,\mathbb{Z}) \to 1 
\end{equation}
where $\cI_g$ (resp. $\cI_g^1$) is called the \textit{Torelli group}. By \cite{birman, powell}, the Torelli group is generated by separating twists and bounding pair maps.

The \textit{Johnson homomorphism} was introduced by Johnson in \cite{Johnson80} to study the action of the Torelli group on the third nilpotent quotient of a surface group. We provide the following characterization. Recall that for any symplectic free $\Z$-module $V$ with symplectic basis $\alpha_i,\beta_i$ ($i=1,\ldots,g$), the form
\begin{equation}
\label{eq:omegaV}
\omega_V= \left(\sum_{i=1}^g \alpha_i \wedge \beta_i\right).
 \end{equation}
does not depend on the choice of symplectic basis. When $V=H$, we simply write $\omega$ for this form. Set $L=\wedge^3H$, and view $H$ as a subgroup of $L$ via the embedding $h\mapsto \omega\wedge h$. The Johnson homomorphism for a once-punctured surface is a group homomorphism $J:\cI_{g}^1 \to L$; by the previous paragraph, it suffices to describe how $J$ operates on separating twists and bounding pair maps. If $T_a$ is a separating twist, then $J(T_a)=0$. Suppose $T_aT_{b}^{-1}$ is a bounding pair map. The curves $a$ and $b$ separate $\Sigma_{g}^1$ into two subsurfaces, let $S$ be the subsurface  which does not contain the puncture. The inclusion $S\hookrightarrow \Sigma_{g}^1$ induces an injective map $H_1(S,\Z) \to H$ which respects the symplectic forms on these spaces. Denote the image of this map by $W$. Then $\omega$ restricts to $\omega_W$ on $W$, and 
\begin{equation}
\label{eq:johnsonBPM}
    J(T_{a}T_{b}^{-1}) = \omega_{W}\wedge [a].
\end{equation}
The Johnson homomorphism for $\Sigma_{g}$ is a homomorphism $J:\cI_g\to L/H$ and operates on separating twists and bounding pair maps as above, except that $S$ may be either of the two subsurfaces cut off by $a$ and $b$.

\subsection{Construction of the Ceresa Class}\label{subsec:constructionCeresa}
In this section, we construct a class in $H^1(\Gamma_g,L/H)$ whose restriction to $\cI_g$
equals twice the Johnson homomorphism. By the work of Hain and Matsumoto \cite{hainmatsumoto}, this class with $\ell$-adic coefficients agrees with the universal Ceresa class over $\mathcal{M}_g$. We discuss this further in \S\ref{subsec:lAdicCeresa}. 

Let $F_{2g} = \pi_1(\Sigma_g^1)$, which is the rank-$2g$ free group, and 
\begin{equation*}
    F_{2g} = L^1F_{2g} \supseteq L^2F_{2g} \supseteq L^3F_{2g} \supseteq \ldots
\end{equation*}
be the lower central series of $F_{2g}$, i.e., $L^{k+1}F_{2g} = [F_{2g},L^kF_{2g}]$.  The $k$-th nilpotent quotient of $F_{2g}$ is $N_k = F_{2g} / L^{k} F_{2g}$. Note that $N_2\cong H$.
Set $\gr_L^k F_{2g} = L^kF_{2g}/L^{k+1}F_{2g}$, which is a central subgroup of $N_{k+1}$. We note that the $N_k$ and the $\gr_L^k F_{2g}$ are all characteristic quotients of $F_{2g}$ and thus carry natural actions of $\Aut(F_{2g})$; in particular, they carry actions of the mapping class group $\Gamma_g^1$.  What's more, the action of $\Aut(F_{2g})$ on $\gr_L^kF_{2g}$ factors through the natural homomorphism $\Aut(F_{2g})  \ra \GL(H)$.  

By \cite[Proposition~2.3]{Morita}, $\Aut(N_3)$ fits into an exact sequence of groups
\begin{equation}
\label{eq:N3Seq}
    1 \to \Hom(H,  \gr_L^2 F_{2g}) \xrightarrow{\phi} \Aut(N_3) \xrightarrow{p} \GL(H) \to 1. 
\end{equation}
Here, for any $f \in \Hom(H,\gr_L^2 F_{2g})$ the action of $\phi(f)$ on $N_3$ is to send $\gamma \in N_3$ to $\gamma f([\gamma])$, where $[\gamma]$ is the image of $\gamma$ under the natural projection to $H$. Because $\Hom(H,  \gr_L^2 F_{2g})$ is abelian, we write the group operation additively.  
The group $\Aut(N_3)$ acts on  $\Hom(H,  \gr_L^2F_{2g})$ by conjugation.

Let $\tau$ be an element of $\Aut(N_3)$ such that $p(\tau) = -I$.  Since $p(\tau)$ is central in $\GL(H)$, any commutator in $\Aut(N_3)$ of the form $[x,\tau]$ lies in $\Hom(H,  \gr_L^2F_{2g})$.  Define 
\begin{equation}
    \label{eq:muTauN3}
    \mu_{\tau}:\Aut(N_3) \to  \Hom(H, \gr_L^2 F_{2g}) \hspace{30 pt} x \mapsto [x,\tau]. 
\end{equation}

\begin{proposition}
\label{prop:muIndependentTau}
    The map $\mu_{\tau}$ is a crossed homomorphism, and its cohomology class 
    \begin{equation*}
        \mu \in H^1(\Aut(N_3), \Hom(H, \gr_L^2 F_{2g}))
    \end{equation*} 
    is independent of the choice of $\tau$.
\end{proposition}

\begin{proof}
That $\mu_\tau$ is a crossed homomorphism follows from the fact that $[xy,\tau] = [y,\tau]^x[x,\tau]$, and hence
\begin{equation*}
    \mu_\tau(xy) = [xy,\tau] = \mu_\tau(x) +  x\cdot \mu_\tau(y).
\end{equation*}
Now suppose we had made a different choice $\tau'$; then $\tau' = t\tau$ for some $t\in  \ker p$. One checks, using the fact that $\ker p$ is abelian, that  
\begin{equation*}
    [x,t\tau] = t^x t^{-1} [x,\tau] 
\end{equation*}
which is to say that \begin{equation*}
    \mu_{\tau'}(x) = \mu_{\tau}(x) - t + x \cdot t  
\end{equation*}
so $\mu_{\tau'}$ is cohomologous to $\mu_\tau$, as claimed.
\end{proof}

The preimage of $\ker p$ under the natural morphism $\Gamma_g^1 \to \Aut(N_3)$ is the Torelli group $\cI_g^1$, and the restriction of this morphism to $\cI_g^1$ is the Johnson homomorphism.  By the work of Johnson \cite{Johnson80}, the homomorphism $J$ is not surjective onto $\Hom(H, \gr_L^2 F_{2g})$; rather, its image is the natural $\GSp(H)$-equivariant inclusion of $L = \wedge^3 H$ into 
\begin{equation*}
    \Hom(H, \gr_L^2 F_{2g}) = \Hom(H, \wedge^2 H).
\end{equation*}
We can thus inflate $\mu$ to $\Gamma_g^1$ to get a cohomology class $\mu \in H^1(\Gamma^1_g, L)$ represented by the cocycle 
\begin{equation}
\label{eq:muTauJohnson}
\mu_{\tau}(\gamma) = J([\gamma,\tau])
\end{equation}
where $\tau \in \Gamma_g^1$ acts on $H$ as $-I$. We say that $\tau$ is a \textit{hyperelliptic quasi-involution}; a hyperelliptic quasi-involution that is an honest involution is called a  \textit{hyperelliptic involution}.
Following Proposition~\ref{prop:muIndependentTau}, the class $\mu$ is defined independent of the choice of $\tau$.

\begin{proposition}
\label{prop:mu2k}
The class $\mu$ is the unique element in $H^1(\Gamma_g^1,L)$ whose restriction to $\cI_g^1$ is $2J \in H^1(\cI_g^1,L)$. 
\end{proposition}

\begin{proof}
Pick an element $\gamma \in \cI_g^1$ and fix a hyperelliptic quasi-involution $\tau\in \Gamma_g^1$. Because $\tau$ acts as $-I$ on $H$, it also acts as $-I$ on $\wedge^3H$. Therefore, we have
\begin{equation*}
 J([\gamma,\tau]) = J(\gamma)+  J(\tau\gamma^{-1} \tau^{-1}) = J(\gamma)+ \tau \cdot J(\gamma^{-1}) = 2J(\gamma). 
\end{equation*}
The uniqueness of $\mu$ follows from \cite[Proposition 5.5]{hainmatsumoto}.
\end{proof}	

Let $\nu \in H^1(\Gamma_g, L/H)$ denote the image of $\mu$ under the composition 
\begin{equation*}
    H^1(\Gamma_g^1,L) \rightarrow H^1(\Gamma_g^1,L/H) \cong  H^1(\Gamma_g,L/H).
\end{equation*}
The map $H^1(\Gamma_g,L/H) \to  H^1(\Gamma_g^1,L/H)$ is induced by restriction, and is an isomorphism by \cite[Proposition 10.3]{hainmatsumoto}. Similar to the once-punctured case, the class $\nu$ is represented by the cocycle $\gamma \mapsto J([\gamma,\tau])$ where $\tau\in \Gamma_g$ is a hyperelliptic quasi-involution.

\begin{definition}
\label{def:ceresaClass}
 The {\em Ceresa class} of $\gamma \in \Gamma_g$ (resp. $\Gamma_g^1$), denoted by $\nu(\gamma)$ (resp. $\mu(\gamma)$), is the restriction of $\nu$ (resp $\mu$) to the cyclic group generated by $\gamma$, viewed as a cohomology class in $H^1(\langle \gamma \rangle,L/H)$ (resp. $H^1(\langle \gamma \rangle,L)$).
\end{definition}

\begin{remark} 
\label{rem:topologicalToEllAdicCeresa}
The justification for this notation is the $\ell$-adic Harris-Pulte theorem of Hain and Matsumoto~\cite[\S\S 8,10]{hainmatsumoto}, which identifies the $\ell$-adic analogue of the classes $\mu,\nu$ defined above with the image of the Ceresa cycle under a cycle class map in \'{e}tale cohomology. We discuss this aspect in detail in \S\ref{subsec:lAdicCeresa}.
\end{remark}

The Ceresa class $\nu(\gamma)$ lies in $H^1(\langle \gamma \rangle, \LH) \cong L / ((\gamma -1)L+H)$. It is certainly trivial for any $\gamma$ which commutes with $\tau$, which is to say it is trivial for any $\gamma$ in the {\em hyperelliptic mapping class group}. 
But the converse is not true; for instance, if $\gamma$ is in the kernel of the Johnson homomorphism, then so is $[\gamma,\tau]$, so $\mu(\gamma) = J([\gamma,\tau]) = 0$; but such a $\gamma$ certainly need not be hyperelliptic. More generally, any mapping class whose commutator with $\tau$ lands in the Johnson kernel has trivial Ceresa class.

When the action of $\gamma$ on $H$ has no eigenvalues that are roots of unity, the group\\ $H^1(\langle \gamma \rangle, \LH)$ is finite and the Ceresa class is torsion.  At the other extreme, if $\gamma$ lies in the Torelli group,  $H^1(\langle \gamma \rangle, \LH) = \LH$ and the Ceresa class is an element of this free $\Z$-module of positive rank and can be of infinite order.

The case of primary interest in the present paper is that where $\gamma$ is a product of commuting Dehn twists, or a {\em positive  multitwist}.  In this case, the action of $\gamma$ on $H$ is unipotent, and $H^1(\langle \gamma \rangle, \LH)$ is infinite; however, in this case, as we shall prove in \S\ref{sec:ceresaClassMultitwist}, the Ceresa class is still of finite order.

\begin{theorem*}[Theorem \ref{thm:nonLagrangianTorsion}]
\label{th:torsion}
Let $\gamma \in \Gamma_g$ be a positive multitwist. Then the Ceresa class $\nu(\gamma)$ is torsion.
\end{theorem*}

In fact, we will show how the order of the Ceresa class can be explicitly computed, though the computation is somewhat onerous.  We will, along the way, prove the analogous statement for multitwists in the punctured mapping class group $\Gamma_g^1$ as well.

Our method for proving Theorem~\ref{th:torsion} will be to show that the Ceresa class lies in a canonical finite subgroup of $H^1(\langle \gamma \rangle, \LH)$, which subgroup we might think of as the tropical intermediate Jacobian.  A notion of tropical intermediate Jacobian was proposed by Mikhalkin and Zharkov in \cite{mikhalkinzharkov}; it would be interesting to know whether the two notions agree in the context considered here.

The next two sections will explain the relationship between the topological, tropical, and local-algebraic pictures; the reader whose interest is solely in the mapping class group can skip ahead to \S\ref{sec-algebra-wedge-H}.

\section{The $\ell$-adic Ceresa class}
\label{sec:connectionToArithmeticGeometry}

In the previous section, we defined the Ceresa class $\nu(\gamma)$ (resp. $\mu(\gamma)$) as a cohomological invariant of any element $\gamma$ in the mapping class group $\Gamma_g$ (resp. $\Gamma_g^1$).  In this section, we discuss how the Ceresa class of multitwists arise in arithmetic geometry. We begin by recalling the monodromy action associated to a 1-parameter family of genus $g$ surfaces degenerating to a stable curve.

\subsection{Monodromy} 
\label{subsec:Monodromy}
Our discussion on the monodromy of a degenerating family of stable curves mainly follows \cite[\S3.2]{AminiBlochBurgosFresan} and \cite[\S1.1]{AsadaMatsumotoOda}. For details on the construction of a local universal family of a stable curve, see \cite[\S3]{TsuchiyaUenoKenji}. Our goal is to recall the non-abelian Picard--Lefschetz formula \cite[Theorem~2.2]{AsadaMatsumotoOda} which says that the monodromy action on the fundamental group of a smooth fiber is given by a multitwist. 

Let $X_0$ be a stable complex curve of genus $g\geq 2$. Let $\cY \to \cD$ be the local universal family for $X_0$ as was constructed in \cite[Theorem~3.1.5]{TsuchiyaUenoKenji}. The base $\cD$ is homeomorphic to $D^{3g-3}$ where $D$ denotes a small open complex disc centered at $0$.  Let $\cB\subset \cD$ be the discriminant locus of $\cY\to \cD$ and $\cD^* =  \cD\setminus \cB$; each fiber $\cY_p$ for $p\in \cD^*$ is diffeomeorphic to a closed  surface $\Sigma_g$. Choose a point $p_0\in \cD^*$ sufficiently close to $0$ at which all loops in $\cD^*$ will be based when we consider its fundamental group. 

The combinatorial data of a stable curve is recorded in its \textit{dual graph}, which is a connected vertex-weighted graph  defined in the following way. Recall that a vertex-weighted graph $\bG$ is a connected graph $G$, possibly with loops or multiple edges, together with a nonnegative integer $w_v$ for each vertex $v$ of $G$. The dual graph $\bG$ of $X_0$ consists of
\begin{itemize}
    \item[-] a vertex $v_{C}$ for each irreducible component $C$ of $X_{0}$
    whose weight is the arithmetic genus of $C$, and 
    \item[-] an edge $e$ between $v_C$ and $v_{C'}$ for each node $n$ in the intersection of $C$ and $C'$.  
\end{itemize}
 For each edge $e_i$ of $\bG$, choose a small loop $\ell_i'$ in the  smooth locus of $X_0$ that goes around $n_i$.  Shrinking $\cD$ if necessary, the inclusion $X_0\hookrightarrow \cY$ admits a retraction $\cY \to X_0$ so that $\cY\to X_0 \to \cY$ is homotopic to the identity. This gives rise to a specialization map $r_p:\cY_p \to X_0$ for each fiber $\cY_p$ over $p \in \cD$. 
Then each $\ell_i = r_p^{-1}(\ell_i')$ defines a closed curve in $\cY_p$.

The discriminant locus $\cB$ is a normal crossings divisor \cite[Proposition~1.1]{AsadaMatsumotoOda} (see also \cite[Theorem~2.7]{Knudsen}). Following \cite[Proposition~1.1(3)]{AsadaMatsumotoOda}, choose coordinates $z_1,\ldots,z_{3g-3}$ on $\cD$ so that $B_i$, the irreducible component of $\cB$ consisting of those $p\in \cD$ such that $\ell_i$ is contractible in  $\cY_p$, has the form 
 $B_i = \{z_i=0\} \cap \cD$.  In particular,  $\cY_0 = X_0$. Then $\cD^*$ is homeomorphic to $(D^*)^N\times (D)^{3g-3-N}$
where $N$ is the number of edges of $G$ and $D^* = D\setminus 0$. Thus $\pi_1(\cD^*)$ is isomorphic to $\oplus_{i=1}^N\Z\cdot \lambda_i$ where $\lambda_i$ is a loop in $\cD^*$ based at $p_0$ that goes around $B_i$. The monodromy action on $\Pi_g$ is given by a non-abelian Picard-Lefschetz formula \cite[Theorem~2.2]{AsadaMatsumotoOda}:
\begin{equation}
\label{eq:nonAbelianMonodromy}
    \rho_{\cY}:\pi_1(\cD^*) \to \Out(\Pi_g) \hspace{30 pt} \lambda_i \mapsto T_{\ell_i}^{-1}.
\end{equation}
Let $\fY \to D$ be a 1-parameter degeneration such that $\fY_0 = X_0$ is the only singular fiber.  Suppose that the local equation in $\fY$ near the node $n_i$ corresponding to the edge $e_i$ of the dual graph is $xy=t^{c_i}$ where $t$ is the parameter on $D$ and $c_i \in \Z_{>0},i=1,\ldots,N$. The following proposition is a variant of \cite[Main~Lemma]{Oda95}.

\begin{proposition}
\label{prop:monodromy}
    The restriction of the monodromy map $\rho_{\cY}$ 
    to $\pi_1(D^*) = \Z\cdot \gamma$  is given by
    \begin{equation*}
    \rho_{\fY}:\pi_1(D^*) \to \Out(\Pi_g) \hspace{30 pt} 
    \gamma \mapsto \prod_{i=1}^N T_{\ell_i}^{-c_i}.
    \end{equation*} 
\end{proposition}

\begin{proof}
The map $\fY \to D$ is given by the pullback of $\cY\to \cD$ under a map $j:D\to \cD$, and $c_i$ is the multiplicity at which $j(D)$ intersects the divisor $B_i$. Explicitly, with $t$ being the local coordinate on $D$, the map $j$ is given by  
\begin{equation*}
j(t) = (a_1t^{c_1},\ldots,a_{N}t^{c_N},0,\ldots,0) + \text{higher order terms}
\end{equation*}
where $a_i\in \C^*$. Composing with the orthogonal projection to $B_i^{\perp} = \{z_k=0\, | \, k\neq i\}$ yields the map  $D\to B_i^{\perp}$ given by $t\mapsto t^{c_i}$, and therefore $\pi_{1}(D^*) \to \pi_1(B_i^{\perp}\setminus 0)$ is given by $\gamma \mapsto c_i\cdot \lambda_i$. The proposition now follows from Equation \eqref{eq:nonAbelianMonodromy}.
\end{proof}

\subsection{The $\ell$-adic Ceresa class for algebraic curves over $\CCt$}
\label{subsec:lAdicCeresa}

In this subsection, we recall the definition of the Ceresa cycle associated with an algebraic curve and its induced class in Galois cohomology, following \cite{hainmatsumoto}. Using comparison theorems, this class agrees with the topological Ceresa class with $\ell$-adic coefficients, justifying the definition of the Ceresa class $\nu$ in \S\ref{sec:topologicalCeresa}.

Let $K$ be a field of characteristic $0$, $G_K$ its absolute Galois group, and $\ell$ a fixed prime number.  Let $X$ be a smooth, complete, genus $g\geq 3$ curve over $K$. For the moment, suppose $X$ has a $K$-rational point $\xi$, which  yields  an embedding $\Phi_{\xi}:X\to \Jac(X)$.  Define algebraic cycles in $\text{CH}_1(\Jac(X))$ given by $X_{\xi}:=(\Phi_{\xi})_*(X)$ and $X_{\xi}^{-}:=(\iota)_*(X_{\xi})$ where $\iota$ is the inverse map on $\Jac(X)$. Because $\iota$ induces the identity map on singular cohomology groups $H_{2r}(\Jac(X),\Z)$ for all $r\ge 0$, the cycle $X_{\xi} - X_{\xi}^{-}$ is null-homologous.
Thus, the image of $X_{\xi} - X_{\xi}^{-}$ under the $\ell$-adic Abel-Jacobi map produces a Galois cohomology class 
\begin{equation*}
\mu^{(\ell)}(X,\xi) \in H^1(G_K, H_{\et}^{2g-3}((\Jac X)_{\Kbar},\, \Zl(g-1))).
\end{equation*}
Via Poincar\'e duality,
\begin{equation*}
H_{\et}^{2g-3}((\Jac X)_{\Kbar},\Z_{\ell}(g-1)) \cong H_{\et}^{3}((\Jac X)_{\Kbar},\Z_{\ell}(1))(-1) \cong 
(\wedge^3H_{\et}^{1}(X_{\Kbar},\Z_{\ell}(1)))(-1).
\end{equation*}
Let $H_{\Zl} = H^1_{\et}(X_{\Kbar}, \Zl(1))$,  $L_{\Zl} = (\wedge^3H_{\Zl})(-1)$, and $\omega\in \wedge^2H_{\Zl}$ the polarization. 

The map $h\mapsto \omega \wedge h$ yields an embedding $H_{\Zl} \hookrightarrow L_{\Zl}$. The \textit{$\ell$-adic Ceresa class}, denoted by $\nu^{(\ell)}(X)$, is the image of $\mu^{(\ell)}(X,\xi)$ under the map $H^1(G_K, L_{\Zl}) \to H^1(G_K, L_{\Zl}/H_{\Zl})$, where we view $\mu^{(\ell)}(X,\xi)$ as an element of $H^1(G_K, L_{\Zl})$. By \cite[\S10.4]{hainmatsumoto}, the class $\nu^{(\ell)}(X)$ only depends on the curve $X/K$ and can be defined when $X$ has no $K$-rational point. Hain and Matsumoto construct a universal characteristic class 

\begin{equation*}
    \nhatell \in H^1(\Out \Pi^{(\ell)}_g, L_{\Zl}/H_{\Zl})
\end{equation*}
which is the $\ell$-adic analog of $\nu$ defined in \S\ref{sec:topologicalCeresa}. 
The class $\nu$ with $\mathbb{Z}_\ell$ coefficient corresponds to $(\hat{\rho}_{X}^{(\ell)})^{*}(\nhatell)$ under the comparison map
\begin{equation*} 
H^1(\Gamma_g,L/H)\otimes \mathbb{Z}_{\ell} \xrightarrow{\sim} H^1_{\text{\'et}}(\mathcal{M}_g\otimes\overline{K},L_{\Zl}/H_{\Zl})
\end{equation*}
where
\begin{equation*}
\hat{\rho}_{X}^{(\ell)}:\pi_1(\mathcal{M}_g,X_{\overline{K}})\to \Out \pi_1^{(\ell)}(X_{\overline{K}})
\end{equation*}
is the universal monodromy representation. Let $n^{(\ell)}(X)\in H^1(G_K,L_{\Zl}/H_{\Zl})$ denote the pullback of $\nhatell$ under the natural action $G_K \to \Out \pi_1^{(\ell)}(X_{\overline{K}})$. The  $\ell$-adic Harris-Pulte theorem \cite[Theorem 10.5]{hainmatsumoto} asserts
\begin{equation}
\label{eq:unpointedComparison}
n^{(\ell)}(X) = \nu^{(\ell)}(X).
\end{equation}
We now show that the $\ell$-adic Ceresa class of a curve over $\CCt$ is torsion.  We obtain this by showing that the $\ell$-adic Ceresa class is, in a natural sense, the $\ell$-adic completion of the Ceresa class of product of Dehn twists attached to the curve in the previous section. This fact justifies calling that topologically-defined class "the Ceresa class," and  allows us to apply Theorem~\ref{thm:nonLagrangianTorsion} to the $\ell$-adic Ceresa class.

\begin{theorem}
\label{thm:elladicCeresa}
    Suppose $K=\CCt$ and $X$ is a smooth curve over $K$. The $\ell$-adic Ceresa class $\nu^{(\ell)}(X)$ is torsion. 
\end{theorem}

\begin{proof}
We begin by reducing to the semistable reduction case. By the semistable reduction theorem, there is a positive integer  $n$ such that the pullback $X'$ of $X$ by the map
\begin{equation*}
    \phi:\Spec K \to \Spec K \hspace{20pt} \phi^{\#}(t) = t^n
\end{equation*}
has semistable reduction. The map $\phi$ induces an endomorphism of $H^1(G_K,L_{\Zl}/H_{\Zl})$ which is multiplication by $n$. This means that $\nu^{(\ell)}(X)$ is torsion if and only if $\nu^{(\ell)}(X')$ is torsion.

Therefore,  we may assume that $X$ has a semistable model $\fX$ defined over $\pO_K$ with special fiber $X_0$. The \'etale local equation in $\fX$ of each node of $X_0$ is $xy=t^{c_i}$ for some $c_i\in \Z_{>0}$. Let 
\begin{equation*}
    S = \Spec \C[\![x_1,\ldots,x_{3g-3}]\!], \hspace{20pt} S' = \Spec \C[\![x_1^{\pm},\ldots, x_{N}^{\pm}, x_{N+1}, \ldots x_{3g-3}]\!].
\end{equation*}
where $N$ is the number of nodes of $X_0$. The map $\fX \to \Spec \pO_K$ is the pullback of the (algebraic) local universal family $\cX \to S$  by a morphism $i: \Spec \pO_K \to S$ of the form
\begin{align*}
    i^{\#}(x_{k}) = a_kt^{c_k} + \text{higher order terms} \hspace{20pt} 
\end{align*}
where $a_k\in \C^*$. Define an analytic family $\fY \to D$ by the pullback of $\cY \to \cD$ under the map 
\begin{equation*}
    j:D\to \cD\hspace{20pt} t\mapsto (a_1t^{c_1},\ldots, a_N t^{c_N},0,\ldots,0).
\end{equation*}
Consider the following diagram.
\begin{equation}
\label{eq:monodromy}
\begin{tikzcd}
 & \pi_1(D^*) \arrow[r, "j_*"] \arrow[d] & \pi_{1}(\cD^*)  \arrow[r, "\rho_{\cY}"] \arrow[d] &   \Out(\Pi_g) \arrow[d] \\
 & G_K \arrow[r, "i_*"] & \pi_{1}^{\alg}(S') \arrow[r, "\rho_{\cX}"] & \Out(\Pi_g^{(\ell)})
\end{tikzcd}
\end{equation}
The left and middle vertical arrows are profinite completions, the right arrow is the $\ell$-adic completion, and $\rho_{\cX}$ is the monodromy map associated to $\cX \to S$. The left square commutes because $G_K\to \pi_1^{\alg}(S')$ is the profinite completion of $\pi_1(D^*) \to \pi_1(\cD^*)$, and the composition of the bottom two arrows may be identified with the natural action $G_K \to \Out(\Pi_g^{(\ell)})$. Commutativity of the right square follows from commutativity of the following diagram
\begin{equation*}
\begin{tikzcd}
 & 1 \arrow[r]& \Pi_g \arrow[r] \arrow[d] & \pi_{1}(\cY^*)  \arrow[r] \arrow[d] & \pi_1(\cD^*) \arrow[r] \arrow[d] & 1 \\
 & 1 \arrow[r]& \widehat{\Pi}_g \arrow[r] & \pi_{1}^{\alg}(\cX^*)  \arrow[r]  & \pi_1^{\alg}(S') \arrow[r]  & 1 
\end{tikzcd}
\end{equation*}
where $\cX^*$ and $\cY^*$ are the smooth loci of $\cX\to S$ and $\cY\to \cD$, respectively.  Here, the vertical arrows are profinite completions and the rows are exact. Let $\gamma \in \Out(\Pi_g)$ denote the image of the counterclockwise generator of $\pi_1(D^*)$ in $\Out(\Pi_g)$. Commutativity of the diagram in \eqref{eq:monodromy} yields the following commutative square
\begin{equation*}
\begin{tikzcd}
 &  H^1(\Out(\Pi_g), L/H) \arrow[r] \arrow[d] &   H^1(\langle \gamma \rangle, L/H) \arrow[d]  \\
 & H^1(\Out(\widehat{\Pi}_g), L_{\Zl}/H_{\Zl}) \arrow[r] &  H^1(G_K, L_{\Zl}/H_{\Zl})
\end{tikzcd}
\end{equation*}
Because the left arrow takes $\nu$ to $\hat{n}^{(\ell)}$, the right arrow takes $\nu(\gamma)$ to $\nu^{(\ell)}(X)$.  By Proposition \ref{prop:monodromy}, $\gamma$ acts as the multitwist $\prod_{i}T_{\ell_i}^{-c_i}$ on $\Pi_g$. The theorem  now follows from Theorem \ref{thm:nonLagrangianTorsion}.
\end{proof}

\begin{remark}
We are indebted to Daniel Litt for the observation that it ought to be possible to prove directly, via arguments on weights \cite{BettsLitt}, that the Ceresa class of a curve over $\CCt$ is trivial, and to derive the topological theorems in this paper from this algebraic fact using the fact that every multitwist can be modeled by an algebraic degeneration.  Our feeling is that modeling the paper this way could create the misleading impression that the topological statement was true for reasons involving hard theorems in algebraic geometry, while in fact, as we shall see, it is a topological fact with a topological proof.
\end{remark}

\section{The tropical Ceresa class}
\label{sec:connectionToTropicalGeometry}

\subsection{Tropical curves}
\label{subsec:tropicalCurves}
A \textit{tropical curve} $\Gamma$ consists of a vertex weighted graph $\bG$, together with a positive real-value $c_e$ associated to each edge $e$, recording its length.  The \textit{genus} of $\Gamma$ is 
\begin{equation*} 
g(\Gamma) = |E(G)| - |V(G)| + 1 +|w|
\end{equation*} 
where $G$ is the underlying graph of $\Gamma$, and $|w|$ is the sum of the vertex weights. This is consistent with the interpretation of a weight on a vertex as an ``infinitesimal loop.'' The \textit{valence} of a vertex $v$, denoted by $\val(v)$, is the number of half-edges adjacent to $v$; in particular, a loop edge contributes $2$ to the valence. 

Given an arrangement of simple, closed, non-intersecting curves $\Lambda = \{\ell_1,\ldots, \ell_N\}$ in $\Sigma_g$, its  \textit{dual graph} is the vertex weighted graph with:
\begin{itemize}
    \item[-] a vertex $v_S$ for each connected component $S$ of  $\Sigma_g\setminus \bigcup_{i=1}^N \ell_i$
    whose weight is the genus of $S$, and 
    \item[-] an edge $e_i$ between $v_S$ and $v_{S'}$ for each loop $\ell_i$ between $S$ and $S'$.  
\end{itemize}
Any vertex-weighted graph $\bG$ of genus $g$ may be realized as the dual graph to an arrangement of pairwise non-intersecting curves on $\Sigma_g$ in the following way. For  each vertex $v$, let $\Sigma_v$ be a genus-$w_v$ surface with $\val(v)$ boundary components.  For each edge $e$ of $\bG$ between the (not necessarily distinct) vertices $u$ and $v$, glue a boundary component of $\Sigma_u$ to a boundary component of $\Sigma_v$; denote the glued locus in the resulting surface by $\ell_e$. This process yields a genus-$g$ surface, together with  an arrangement of pairwise non-intersecting curves $\Lambda = \{\ell_e \, | \, e\in E(\Gamma)\}$  whose dual graph is $\bG$.  For an illustration, see Figure \ref{fig:graphToSurface}. 
If $\Gamma$ is a tropical curve with integral edge lengths $c_e$, then we have a canonical multitwist
\begin{equation*}
T_{\Gamma} = \prod_{e\in E(\Gamma)} T_{\ell_e}^{c_e}
\end{equation*}
\noindent and we let $\delta_{\Gamma}\in \Sp(H)$ denote the action of $T_{\Gamma}$ on $H$. At this point, one may define the Ceresa class of $\Gamma$ to be $\nu(T_{\Gamma})\in H^1(\langle T_{\Gamma} \rangle, L/H) = H^1(\langle \delta_{\Gamma} \rangle, L/H)$.  However, when the edge lengths of $\Gamma$ are not integral, we cannot define $\delta_{\Gamma}$ and $\nu(\Gamma)$ in terms of a multitwist. 
We will define the Ceresa class for a tropical curve and what it means for it to be trivial in \S\ref{subsec:tropicalCeresa}. 

\begin{figure}[tbh]
	\centering
	\includegraphics[height=35mm]{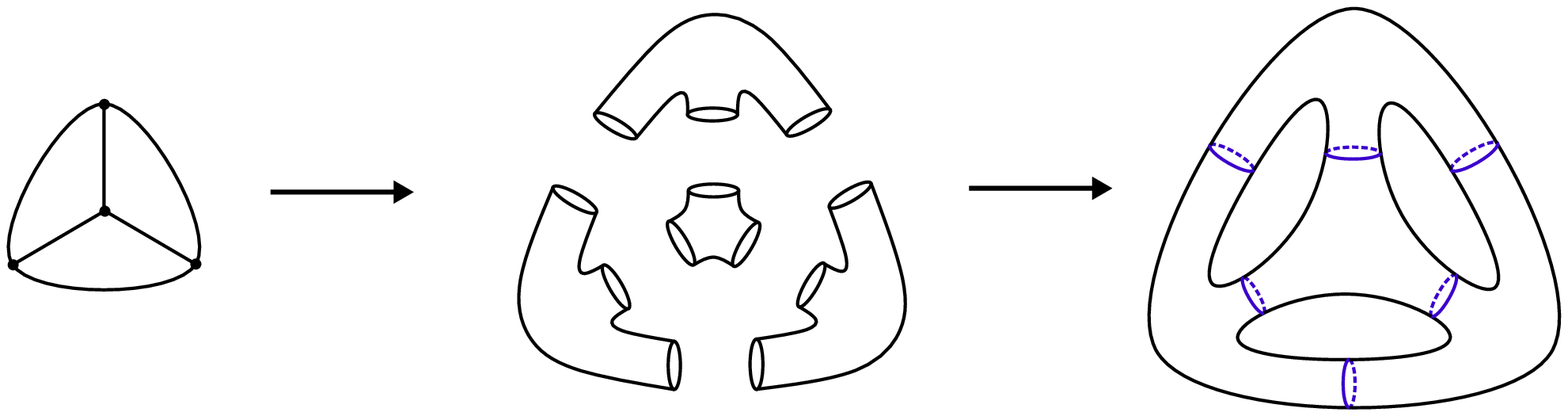}
	\caption{Passing from a tropical curve to a multitwist} \label{fig:graphToSurface}
\end{figure}

\subsection{The tropical Jacobian} 
\label{subsec:TropicalJacobian}
Now suppose $\Gamma$ has genus $g \geq 2$ and  fix an orientation on the underlying graph $G$.  Its \textit{Jacobian} is the real $g$ dimensional torus 
\begin{equation*}
    \Jac(\Gamma) = (H_1(G,\R) \oplus \R^{|w|})/(H_1(G,\bZ) \oplus \bZ^{|w|}) 
\end{equation*}
together with the semi-positive quadratic form $Q_{\Gamma}$ which vanishes on $\R^{|w|}$ and on $H_1(G,\R)$ is equal to 
\begin{equation*}
	Q_{\Gamma} \left( \sum_{e\in E(G)} \alpha_{e}\cdot e  \right) = \sum_{e\in E(G)} \alpha_e^2\cdot c_e. 
\end{equation*}
The form $Q_{\Gamma}$ is positive definite when  all vertex weights are $0$,  and $\det(Q_{\Gamma})$  is the \textit{first Symanzik polynomial} of of $G$  \cite[Proposition~2.9]{AminiBlochBurgosFresan}. That is,
\begin{equation}
\label{eq:SymanzikPolynomial}
    \det(Q_{\Gamma}) = \sum_{T} c_{T}, \hspace{15pt} \text{where} \hspace{15pt} c_T = \prod_{e\notin E(T)} c_e
\end{equation}
and the sum is taken over all spanning trees $T$ of $G$.

When $\Gamma$ has integral edge lengths,  $\delta_{\Gamma}$ and  $Q_{\Gamma}$ are related in the following way. First, embed $G$ into $\Sigma_g$ so that each vertex $v$ maps to a point in $\Sigma_v$, and each edge $e$ maps to a simple arc intersecting the loop $\ell_{e}$ exactly one time, and no other $\ell_{e'}$. This embedding, which we denote by  $\iota:G\hookrightarrow \Sigma_g$, induces an injective map on integral homology groups. Then 
\begin{equation}
\label{eq:deltaToQ}
Q_{\Gamma}(\gamma) = \hat{i}( [\iota(\gamma)], (\delta_{\Gamma}-I)[\iota(\gamma)]).
\end{equation}

Here is a more explicit description of the relationship between $Q_{\Gamma}$ and $\delta_{\Gamma}$.  Enumerate the edge set $E(G) = \{e_1,\ldots,e_N\}$ so that $E(G)\setminus \{e_1,\ldots,e_h\}$ are the edges of a spanning tree $T$.  The graph $T\cup \{e_i\}$ has a unique cycle, denote by $\gamma_i$ the image of this cycle under $\iota$.  The cycles $[\gamma_1],\ldots,[\gamma_h]$ form a basis for an isotropic subspace of $H$. Orient $\gamma_i$ and $\ell_{e_i}$ so that $\hat{i}([\gamma_i], [\ell_{e_i}]) = 1$, for $1\leq i\leq h$.
Setting $\alpha_i = [\gamma_{i}]$ and $\beta_i = [\ell_{e_i}]$ yields a symplectic basis of a symplectic subspace of $H$.  This extends to a symplectic basis $\alpha_1,\ldots,\alpha_g,\beta_1,\ldots,\beta_g$ on all of $H$, allowing us to identify $Q_{\Gamma}$ with a symmetric $g\times g$ matrix. Then
\begin{equation}
\label{eq:QTodelta}
\delta_{\Gamma} = \begin{pmatrix}
I & 0 \\
Q_{\Gamma} & I
\end{pmatrix}.
\end{equation}
In particular, we may identify $Q_{\Gamma}$ with the restriction $\delta_{\Gamma}-I:H/Y \to Y$, where $Y$ is the $\Z$-submodule of $H$ spanned by the $\beta_i$ for $i=1,\ldots,h$.

\subsection{The tropical Ceresa class}
\label{subsec:tropicalCeresa}
We saw in \S\ref{subsec:tropicalCurves} how one may define the Ceresa class of a tropical curve with integral edge lengths in terms of a multitwist. When the edge lengths of $\Gamma$ are not integral, then we do not have access to such a multitwist. 
Instead, we will define what it means for a tropical curve to be Ceresa trivial.

The kernel of the Johnson homomorphism, denoted by $\cK_g$, is a normal subgroup of $\Gamma_g$, which allows us to form the quotient $\cG_g = \Gamma_g/\cK_g$. This follows from the fact that $\cK_g$ is the kernel of the map $\Gamma_g \to \Out(N_3)$. Let $\bG$ be a vertex-weighted graph, and denote by $\pB(\bG)$ the subgroup of $\cG_g$ generated by the twists $T_{\ell_e}$ for $e\in E(\bG)$. This is a free $\Z$-module because  the $\ell_e$'s are non-intersecting, and it has rank $N-s$ where $s$ is the number of separating edges in $\bG$. Given a hyperelliptic quasi-involution $\tau \in \cG_g$, define
\begin{equation*} 
\pB_{\tau}(\bG) = \pB(\bG) \cap C_{\cG_g}(\tau) \leq \pB(\bG)
\end{equation*} 
where $C_{\cG_g}(\tau)$ denotes the centralizer of $\tau$ in $\cG_g$.  Let
\begin{equation*}
\nu(\Gamma) = \sum_{e\in E(\bG)} c_e T_e \in \pB(\bG)_{\R}.
\end{equation*}
We say that $\Gamma$ is \textit{Ceresa trivial} if there exists a hyperelliptic quasi-involution $\tau$ such that  $\nu(\Gamma) \in \pB_{\tau}(\bG)_{\R}$. Proposition \ref{prop:tropicalCeresaIntegralEdge} below shows that this notion agrees with the Ceresa class associated to the multitwist $T_\Gamma$ being trivial in the case where $\Gamma$ has integral edge length, but first we will need the following Lemma.

\begin{lemma}
\label{lem:ceresaTrivialRep}
The Ceresa class $\nu(T_{\Gamma})$ is trivial if and only if there is a hyperelliptic quasi-involution $\tau$ such that $J([T_{\Gamma},\tau]) = 0$. 
\end{lemma}

\begin{proof}
The ``if'' direction is clear. Suppose $\nu(T_{\Gamma}) = 0$.  The class $\nu(T_{\Gamma})$ is represented by the cocycle $\gamma \mapsto J([\gamma,\tau])$ for some hyperelliptic quasi-involution $\tau$, and hence
\begin{equation*}
    J([T_{\Gamma},\tau]) = T_{\Gamma}\cdot h-h
\end{equation*}
for some $h\in L/H$. Because the Johnson homomorphism is surjective, there is a $t\in \cI_g$ such that $J(t)=h$. By rearranging the above equality, we see that 
 $J([\gamma,t^{-1}\tau]) = 0$. The Lemma now follows from the fact that $t^{-1}\tau$ is also a hyperelliptic quasi-involution.  
\end{proof}

\begin{proposition}
\label{prop:tropicalCeresaIntegralEdge}
    Suppose $\Gamma$ has integral edge lengths. Then $\Gamma$ is Ceresa trivial if and only if $\nu(T_{\Gamma}) = 0$. 
\end{proposition}

\begin{proof}
The tropical curve $\Gamma$ is Ceresa trivial if and only if there is a hyperelliptic quasi-involution  $\tau$ such that $\nu(\Gamma)\in \pB_{\tau}(\bG)$, if and only if $[T_{\Gamma},\tau] = 1$ in $\cG_g$, if and only if 
$J([T_{\Gamma},\tau]) = 0$ in $L/H$, if and only if   $\nu(T_{\Gamma}) = 0$.  The last equivalence is Lemma \ref{lem:ceresaTrivialRep}.
\end{proof}

\begin{proposition}
\label{prop:ceresaRealEdgelengths}
The following are equivalent: 
\begin{enumerate}
    \item  $\Gamma$ is Ceresa nontrivial for all positive real edge lengths;
    \item $\Gamma$ is Ceresa nontrivial for all positive integral edge lengths;
    \item $\nu(T_{\Gamma}) \neq 0$ for all positive integral edge lengths;
    \item for any hyperelliptic quasi-involution $\tau$, the subgroup $\pB_{\tau}(\bG)$ is contained in a coordinate hyperplane of $\pB(\bG)_{\R}$. 
\end{enumerate}
\end{proposition}

\begin{proof}
The implications $(4) \Rightarrow (1) \Rightarrow (2)$ are clear,  and   $(2) \Rightarrow (3)$ follows from Proposition \ref{prop:tropicalCeresaIntegralEdge}. Suppose there is a $\tau$ such that $\pB_{\tau}(\bG)$ is not contained in a coordinate hyperplane of $\pB(\bG)_{\R}$. This means that there is a lattice point in $\pB_{\tau}(\bG)$ whose coordinates are all positive. This corresponds to a tropical curve $\Gamma$ with underlying vertex-weighted graph $\bG$ such that $\nu(T_{\Gamma}) = 0$. This proves $(3) \Rightarrow (4)$. 
\end{proof}

We end this section by showing that the Ceresa class vanishes for hyperelliptic tropical curves. First, we recall some terminology. Let $\bG$ be a vertex-weighted graph. A vertex $v$ of a vertex-weighted graph $\bG$ is \textit{stable} if 
\begin{equation*}
    2w_v-2+\val(v) >0.
\end{equation*}
 and $\bG$ is stable if all of its vertices are stable. A tropical curve is stable if its underlying weighted graph is stable. 
Two tropical curves are \textit{tropically equivalent} if one can be obtained from the other via a sequence of the following moves:
		\begin{itemize}
			\item[-] adding or removing a 1-valent vertex $v$ with $w_v = 0$ and the edge incident to $v$, or
			\item[-] adding or removing a 2-valent vertex $v$ with $w_v = 0$, preserving the underlying metric space.
		\end{itemize}
Every tropical curve $\Gamma$ of genus $g\geq 2$ is tropically equivalent to a unique tropical curve whose underlying weighted graph is stable \cite[Section~2]{Caporaso}. 

\begin{lemma}
\label{lem:tropEquivalentCeresa}
If $\Gamma$ and $\Gamma'$ are tropically equivalent, then $\nu(\Gamma) = \nu(\Gamma')$.
\end{lemma}
\begin{proof}
Let $v$ be a vertex with $w_v=0$. Suppose  $\val(v)=1$, and denote by $e$ the adjacent edge.  Then $\Sigma_v$ is a disc, so $\ell_e$ is contractible, and hence $T_{\ell_e} =1$.  Now suppose $\val(v)=2$, and denote by $e,f$ the adjacent edges. Then $\ell_{e}$ is isotopic to $\ell_{f}$, and hence $T_{\ell_{e}} = T_{\ell_{f}}$. We conclude that the Ceresa class of tropically equivalent tropical curves coincide. 
\end{proof}

Suppose $\Gamma$ is a stable hyperelliptic tropical curve with underlying vertex-weighted graph $\bG$, and $\sigma$ the hyperelliptic involution of $\Gamma$. That is, $\sigma:\Gamma \to \Gamma$ is an isometry that induces an involution of $\bG$  such that all vertices of positive weight are fixed and  $\Gamma/\sigma$ is a metric tree.  By \cite[Proposition~2.5]{Corey} the edge set of $\Gamma$ partitions into the subsets
\begin{itemize}
    \item[-] $\{e\} $ for separating edges $e$ and $\sigma$ restricted to $e$ is the identity,
    \item[-] $\{e,f\}$ where $e\neq f$ form a separating pair of edges and $\sigma(e) = f$, and
    \item[-] $\{e\}$ where $e$ is any other edge, and $\sigma$ takes $e$ to itself, interchanging its endpoints. 
\end{itemize}
If $\{e,f\}$ is a separating pair, then $c_{e} = c_f$ because $\sigma$ is an isometry.

\begin{lemma}
\label{lem:tauHyperelliptic}
Suppose $\Gamma$ is a 2-edge connected stable hyperelliptic tropical curve, and $\{\ell_{e} \, | \, e\in E(\Gamma)\}$ is an arrangement of loops on $\Sigma_g$ whose dual graph is $\Gamma$.  There is a hyperelliptic quasi-involution $\tau$ of $\Sigma_g$ such that $\tau(\ell_e) = \ell_{\sigma(e)}$.  
\end{lemma}

\begin{remark} Note that the quasi-involution cannot in general be taken to be an involution; this means that the proof is necessarily more complicated than showing that a hyperelliptic involution of the graph lifts in some natural (functorial) way to an involution of $\Sigma_g$.  On the other hand, if $\Gamma$ is 2-vertex connected, so that (in the language of the proof) there is only one $\Sigma_i$, the quasi-involution we construct is in fact an involution.
\end{remark}

\begin{proof}

Let $\Gamma_1,\ldots,\Gamma_k$ be a block decomposition of $\Gamma$ in the sense of \cite[\S 2]{Corey}. The hyperelliptic involution $\sigma$ restricts to a hyperelliptic involution on each $\Gamma_i$ because $\sigma$ fixes vertices of positive weight and acts as $-I$ on $\Jac(\Gamma)$ \cite[Theorem~5.19]{BakerNorine}.
If $\Gamma_i$ is a single vertex of weight 1, then let $\Sigma_i$ be a genus-1 surface with one boundary component, and $\tau_i:\Sigma_i\to \Sigma_i$ be a orientation-preserving homeomorphism that acts as $-I$ on $H_1(\Sigma_i,\Z)$ and restricts to the identity on $\partial \Sigma_i$. If $\Gamma_i$ is a single vertex with a loop edge $e$, then let $\Sigma_i$ be a genus-1 surface with one boundary component, and $\tau_i:\Sigma_i\to \Sigma_i$ be a orientation-preserving homeomorphism that acts as $-I$ on $H_1(\Sigma_i,\Z)$, $\tau(\ell_e) = \ell_e$, and restricts to the identity on $\partial \Sigma_i$. 

Otherwise, $\Gamma_i$ is 2-vertex connected and genus $g_i\geq 2$.   Form $\Sigma_{g_i}$ as in \S4.1. For each $u\in V(\Gamma_i)$ fixed by $\sigma$, remove a small open disc $S_u$ from $\Sigma_u^{\circ}$; denote resulting boundary curve by $\ell_u$ and the resulting surface by $\Sigma_i$. For each $u\in V(\Gamma_i)$ (resp. $e\in E(\Gamma)$) choose a point $p_u\in \Sigma_u^{\circ}$ (resp. $p_e\in \ell_e$). For each half-edge $h$ of $\Gamma_i$, define a simple path $\eta_h$ in $\Sigma_i$ satisfying the following. 
\begin{itemize}[noitemsep]
    \item[-] If $\sigma(u)\neq u$, and  $h$  is a half-edge of $e$ adjacent to $u$, then $\eta_h$ is a simple path in $\Sigma_u$ from $p_u$ to $p_e$ meeting $\partial \Sigma_u$ only at $p_e$. 
    \item[-] If $\sigma(u)= u$, and $h$  is a half-edge of $e$ adjacent to $u$, then $\eta_h$ is a simple path in $\Sigma_u\setminus S_u$ from $p_u$ to $p_e$ meeting $\partial (\Sigma_u\setminus S_u)$ only at $p_e$. 
\item[-] If $h,h'$ are adjacent to $u$, then $\eta_{h} \cap \eta_{h'} = p_u$.
\end{itemize}
We claim that there are orientation-preserving homeomorphisms $\tau_u:\Sigma_u \to \Sigma_{\sigma(u)}$ so that 
	\begin{itemize}[noitemsep]
	\item[-] $\tau_u(\eta_h) = \eta_{\sigma(h)}$,
	\item[-] $\tau_u|\ell_e = \tau_v|\ell_e$, and
	\item[-] the restriction of $\tau_u$ to $\ell_u$ is the identity, for each $u\in V(\Gamma_i)$ fixed by $\sigma$. 
	\end{itemize}
Suppose $\sigma(u) = v\neq u$. Order the half edges of $u$ (resp. $v$) by $h_1,\ldots,h_a$ (resp. $k_1,\ldots,k_a$) such that $\sigma(h_j) = k_j$, and denote by $e_j$ the edge containing $h_j$ (resp. $f_j$ the edge containing $k_j$). Let $D$ be an oriented $3a$-gon, and label the edges of $D$ (counterclockwise) by $\eta_{h_1},\ell_{e_1},\eta_{h_1}^{-1},\ldots, \eta_{h_a},\ell_{e_a},\eta_{h_a}^{-1}$. Gluing  $\eta_{h_j}$ along $\eta_{h_j}^{-1}$ (for $j=1,\ldots,a$) yields a quotient map $\pi_u:D \to \Sigma_u$, see Figure \ref{fig:gluingSurface} for an illustration. Now relabel the edge $\eta_{h_j}$ by $\eta_{k_j}$, $\ell_{e_j}$ by $\ell_{f_j}$, and $\eta_{h_j}^{-1}$ by $\eta_{k_j}^{-1}$.  Gluing $\eta_{k_j}$ along $\eta_{k_j}^{-1}$ (for $j=1,\ldots,a$) yields a quotient map $\pi_v:D \to \Sigma_v$. This map induces a homeomorphism on the quotient $\tau_u:\Sigma_u\to \Sigma_v$ such that $\tau_u(\eta_{h_j}) = \eta_{k_j}$ and $\tau_u(\ell_{e_j}) = \tau_u(\ell_{f_j})$. In particular, the composition $\tau_{v}\tau_{u}$ is the identity on $\Sigma_u$. If $\Gamma$ is 2-vertex connected and stable, it has no vertices fixed by $\sigma$, and we may glue these $\tau_u$ to give a hyperelliptic involution $\tau$.

\begin{figure}[tbh!]
    \centering
    \includegraphics[width=10cm]{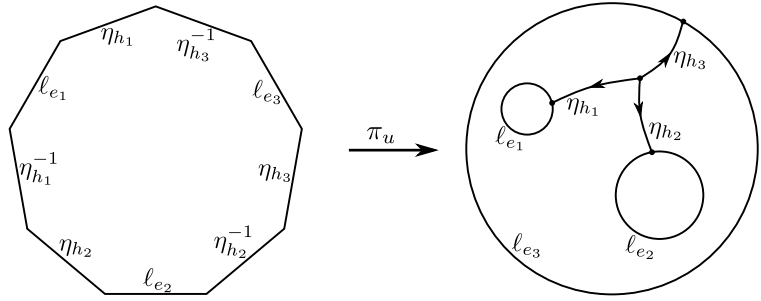}
    \caption{The gluing map $\pi_u:D\to \Sigma_u$ when $a=3$}
    \label{fig:gluingSurface}
\end{figure}

Now suppose $\sigma(u) = u$. Label the half-edges of $u$ by $h_1,\ldots,h_a,k_1,\ldots,k_a$ such that $\sigma(h_j) = k_j$, and denote by $e_j$ the edge containing $h_j$ (resp. $f_j$ the edge containing $k_j$). Let $\eta_u$ be a simple path in $\Sigma_u\setminus S_u$ from $p_u$ to a point $p$ on $\ell_u$ that meets the other $\eta_{h_i}$'s and $\eta_{k_i}$'s only at $p_u$. Let $D$ be an oriented $6a+3$-gon. Label the edges of $D$ (counterclockwise) by
\begin{equation*}
    \eta_{h_1},\ell_{e_1},\eta_{h_1}^{-1},\ldots, \eta_{h_a},\ell_{e_a},\eta_{h_a}^{-1}, \eta_{u},\ell_u,\eta_{u}^{-1},\eta_{k_a},\ell_{f_a},\eta_{k_a}^{-1},\ldots, \eta_{k_1},\ell_{f_1},\eta_{k_1}^{-1}.
\end{equation*} 
Gluing  $\eta_u$ along $\eta_{u}^{-1}$, $\eta_{h_j}$ along $\eta_{h_j}^{-1}$, and $\eta_{k_j}$ along $\eta_{k_j}^{-1}$ (for $j=1,\ldots,a$) yields a quotient map $\pi_u:D \to \Sigma_u$. 
Now relabel $\eta_{h_j}$ by $\eta_{k_j}$, $\eta_{h_j}^{-1}$ by $\eta_{k_j}^{-1}$, and $\ell_{e_j}$ by $\ell_{f_j}$ ($j=1,\ldots,a$).  Gluing  $\eta_u$ along $\eta_{u}^{-1}$, $\eta_{h_j}$ along $\eta_{h_j}^{-1}$, and $\eta_{k_j}$ along $\eta_{k_j}^{-1}$ (for $j=1,\ldots,a$) yields another quotient map $\pi_u':D \to \Sigma_u$. This map induces a homeomorphism on the quotient $\tau_u:\Sigma_u\to \Sigma_u$ such that $\tau_u(\eta_{h_j}) = \eta_{k_j}$, $\tau_u(\ell_{e_j}) = \ell_{f_j}$, (for $j=1,\ldots,a$),  and is the identity on $\ell_u$.  

Finally, we may modify the $\tau_u$'s in a collar neighborhood of $\partial \Sigma_u$ so that $\tau_u|\ell_e = \tau_v|\ell_e$ when $e$ is an edge between $u$ and $v$.  Having done so, the resulting $\tau_u$'s glue to give an orientation-preserving homeomorphism $\tau_i:\Sigma_i\to \Sigma_i$ that restricts to the identity on $\partial \Sigma_i$, and which sends $\ell_e$ to $\ell_{\sigma(e)}$ for all edges $e$. 

Define a homology basis of $\Sigma_{g_i}$ in the following way. Let $e_1,\ldots, e_{g_i}\in E(\Gamma_i)$ be a collection of edges whose removal from $\Gamma_i$ is a spanning tree $T$.  Denote by $h_j$ the unique cycle in $T\cup \{e_j\}$. Let $\gamma_i$ be the simple closed curve in $\Gamma_{g_i}$ formed by the paths $\eta_h$ for all half-edges in the path $h_j$. Orient $\ell_i$ and $\gamma_i$ so that $\hat{i}([\gamma_i],[\ell_{e_i}]) = 1$; the cycles \begin{equation*}
    [\ell_{e_1}]\ldots, [\ell_{e_{g_i}}],  [\gamma_1],\ldots, [\gamma_{g_i}]
\end{equation*}
form a symplectic basis of $H_1(\Sigma_{g_i},\Z)$.

Next, we claim, for $j=1,\ldots,g_i$, that 
\begin{align}
& \tau_*([\ell_{e_j}]) = -[\ell_{e_j}], \label{eq:tauEll}\\
& \tau_*([\gamma_j]) = -[\gamma_j]. \label{eq:tauGammai}
\end{align} 
Consider \eqref{eq:tauEll}. Denote the vertices of $e_{j}$ by $u_j$ and $v_j$. Without loss of generality, suppose $\Sigma_{u_j}$ is on the left of $\ell_{e_j}$.  Because $\tau_{u_j}$ is orientation preserving, $\Sigma_{v_j}$ appears on the left of $\tau(\ell_{e_j})$.  If $e_{j}$ is flipped, then $\tau(\ell_{e_j}) = \ell_{e_j}$, and because $\Sigma_u$ appears on the left of $\ell_e$ but on the right of $\tau(\ell_e)$, we have $\tau_*([\ell_{e_j}]) = -[\ell_{e_j}]$.  Now suppose $\{e_j,f_j\}$ is a separating pair, orient $\ell_{f_j}$ so that $[\ell_{e_j}] = [\ell_{f_j}]$.  Their removal splits $\Sigma_{g_i}$ into two surfaces $S,S'$ with boundary curves $\ell_{e_j},\ell_{f_j}$. The subsurfaces  $\Sigma_{u_j}$ and $\Sigma_{v_j}$ belong to the same surface; suppose it is $S$. Because $S$ lies on the left of both $\ell_{e_j}$ and $\tau(\ell_{e_j})$, we must have that $\tau_*([\ell_{e_j}]) = -[\ell_{f_j}] = -[\ell_{e_j}]$, and therefore \eqref{eq:tauEll}.

Now consider \eqref{eq:tauGammai}. Because $\tau_i$ is orientation-preserving, we have 
\begin{equation*}
    \hat{i}(\tau_*([\gamma_j]), \tau_*([\ell_{e_j}])) = \hat{i}(\tau_*([\gamma_j]), -[\ell_{e_j}]) = 1.
\end{equation*}
Together with the fact that $\tau_i(\gamma_j) = \gamma_j$, we have $\tau_{*}([\gamma_j]) = -[\gamma_j]$.

Finally, we will piece together the $\tau_i$'s to get the requisite hyperelliptic quasi-involution. For each cut-vertex $u$ of $\Gamma$, let $\Sigma_{u}$ be a genus-0 surface with $n_u$ boundary components, where $n_u$ is the number of blocks attached to $u$.  For each block $\Gamma_i$ attached at $u$, glue $\Sigma_i$ to $\Sigma_u$ along the corresponding boundary components. The orientation-preserving homeomorphism  $\tau:\Sigma_g\to \Sigma_g$ given by 
\begin{align*}
    \tau|\Sigma_{i} &= \tau_i \hspace{5pt} \text{ for } \hspace{5pt} i=1,\ldots,k \\
    \tau|\Sigma_u &= \id_{\Sigma_u} \hspace{5pt} \text{ if $u$ is a cut-vertex of } \Gamma
\end{align*}
acts as $-I$ on $H_{1}(\Sigma_g,\Z)$ and satisfies $\tau(\ell_e) = \ell_{\sigma(e)}$ for all $e\in E(\Gamma)$, as required.

\end{proof}

\begin{proposition}
\label{prop:hyperellipticCeresaTrivial}
    If $\Gamma$ is a hyperelliptic tropical curve, then $\Gamma$ is Ceresa trivial.  
\end{proposition}

\begin{proof}
By Lemma \ref{lem:tropEquivalentCeresa}, we may assume that $\Gamma$ is stable. Denote by $\bG^2$ (resp. $\Gamma^2$) the 2-edge connectivization of $\bG$ (resp. $\Gamma$) (this is obtained by contracting all separating edges, see \cite[Definition~2.3.6]{CaporasoViviani}). Because $\pB(\bG) = \pB(\bG^2)$  and $\nu(\Gamma)=\nu(\Gamma^2)$, we may assume that $\Gamma$ is 2-edge connected.  Let $\tau$ be the hyperelliptic involution from  Lemma \ref{lem:tauHyperelliptic}. If $e$ is a separating edge, then $T_{\ell_e}$ is a separating twist, which is trivial in $\pB(\bG)$. If $\{e,f\}$ is a pair of separating edges, then $T_{\ell_e} + T_{\ell_f} \in \pB_{\tau}(\bG)$. If $e$ is any other edge, then $T_{\ell_e}\in \pB_{\tau}(\bG)$.  Decompose $\nu(\Gamma)$ as 
\begin{equation*}
\nu(\Gamma) = \sum (c_eT_{\ell_e} + c_fT_{\ell_f}) + \sum c_eT_{\ell_e}    
\end{equation*}
where the sum on the left is over all separating pairs, and the sum on the right is over all non-separating edges not in a separating pair.  Because $c_e=c_f$ whenever $\{e,f\}$ is a separating pair, we have that  $\nu(\Gamma)\in \pB_{\tau}(\bG)_{\R}$, and hence $\Gamma$ is Ceresa trivial. 
\end{proof}

\section{A finite subgroup of $H^1(\Z,L)$}
\label{sec-algebra-wedge-H}

\subsection{Filtrations on $H^1(\Z,\wedge^kH)$}
\label{subsec:FiltrationsOnH1} 

In this section we set up a rather general framework for abelian groups with a unipotent automorphism, which we will apply in the case of the singular homology of a genus-$g$ topological surface acted upon by a multitwist.

Let $H$ be a finitely generated free $\mathbb{Z}$-module and $\delta\in \SL(H)$ be an element such that $(\delta-I)^2 = 0$.  We consider the action of the cyclic group $\langle\delta\rangle \cong \Z$ on $H$, which induces an action of $\langle \delta \rangle$ on $\wedge^k H$ for any $k \geq 0$.
Let $Y$ be the saturation of $\image(\delta-I)$ in $H$. By the hypothesis on $\delta$, we have $Y\leq \ker(\delta-I)$, i.e., $\delta$ acts trivially on $Y$.   Consider the following descending filtration on $\wedge^kH$:
\begin{equation}
\label{eq:Ffiltration}
F_{q}\wedge^kH = (\wedge^qY) \wedge (\wedge^{k-q} H). 
\end{equation}
Note that $F_{q}\wedge^kH$ is saturated in $\wedge^kH$, so the graded piece $\gr_q^F\wedge^kH := F_{q}\wedge^kH / F_{q+1}\wedge^kH$ is torsion-free. The following Lemma shows that $(\delta-I)(F_{q}\wedge^kH) \leq F_{q+1}\wedge^kH$ for any $k \ge q \ge 0$. 
\begin{lemma}
	\label{lem:deltaISimpleWedge}
	For $y\in \wedge^{q}Y$ and $h=h_1\wedge\ldots\wedge h_{k-q} \in \wedge^{k-q}H$, 
	\begin{equation*}
	(\delta-I)(y \wedge h) = \sum_{i=1}^{k-q} y\wedge h_1 \wedge \ldots \wedge (\delta-I) h_i \wedge \ldots \wedge h_{k-q} \mod F_{q+2}\wedge^k H.
	\end{equation*} 
	In particular, $(\delta-I)(F_q\wedge^kH) \leq F_{q+1}\wedge^kH$.
\end{lemma}

\begin{proof}
Because $\delta y= y$, we can write
\begin{equation*} 
(\delta - I)(y\wedge h) = y\wedge (\delta - I + I)h_1 \wedge \ldots \wedge (\delta - I + I)h_{k-q} - y\wedge h
\end{equation*}
and expand the latter as
\begin{align*}
 &\sum_{i=1}^{k-q}  y \wedge h_1 \wedge \ldots \wedge (\delta -I) h_i \wedge \ldots \wedge h_{k-q} \\
&+ \sum_{1\leq i<j\leq k-q}  y \wedge h_1 \wedge \ldots \wedge (\delta -I) h_i \wedge \ldots \wedge (\delta-I) h_j \wedge \ldots \wedge h_{k-q} + \ldots \qedhere
\end{align*}
\end{proof}

This means, in particular, that $\delta-I$ induces a map  $\gr_{q-1}^F\wedge^kH \ra \gr_{q}^F\wedge^kH$ for all $q$.  While these maps are typically not surjective, what we will see in the lemma below is that at least half of them are surjective {\em rationally}; that is, their cokernels are finite.

\begin{lemma}
	\label{lem:GradedSurjective}
	The map
	\begin{equation*}
	(\delta-I)_{\Q}: \gr_{q-1}^F\wedge^kH_{\Q} \ra \gr_{q}^F\wedge^kH_{\Q}
	\end{equation*}
	is surjective for  $q > k/2$.
\end{lemma}
	
\begin{proof}
Let $y = y_1 \wedge \ldots \wedge y_{q} \in \wedge^qY_{\Q}$  and $h =  h_1 \wedge \ldots \wedge h_{k-q} \in \wedge^{k-q}H_{\Q}$.  It suffices to show that $y \wedge h$ lies in the image of $(\delta-I)_\Q\bmod F_{q+1}\wedge^kH_{\Q}$ for all such $y,h$. Choose $x_i\in H_{\Q}$ such that $(\delta-I) x_i = y_i$. For $I \in {[q] \choose p}$, let $\hat{y}_{I}$ be obtained from $y$ by replacing $y_{i}$ with $x_{i}$ for each $i \in I$.  Similarly, for $J \in {[k-q]\choose p}$  let $\hat{h}_J$ be obtained from $h$ by replacing $h_{j}$ with $(\delta-I) h_{j}$ for each $j\in J$.  We define 
\begin{equation*}
\mathfrak{y}(a,b) = \sum_{I\in  {[q] \choose a},J \in  {[k-q] \choose b}} \hat{y}_{I} \wedge \hat{h}_{J} \in F_{q-a+b}\wedge^kH_{\Q},
\end{equation*}
Note that $\mathfrak{y}(a,b) \neq 0$ only if $0 \leq a\leq q$ and $0\leq b \leq k-q$, and $\mathfrak{y}(0,0) = y \wedge h$.
By Lemma \ref{lem:deltaISimpleWedge}, 
\begin{equation*}
(\delta-I)_{\Q} (\hat{y}_{I} \wedge \hat{h}_{J}) \equiv 
\sum_{i\in I} \hat{y}_{I\setminus\{i\}} \wedge \hat{h}_{J} + 
\sum_{j\notin J} \hat{y}_{I} \wedge \hat{h}_{J\cup \{j\}} \bmod F_{q-a+b+2}\wedge^kH_{\Q}.
\end{equation*}
In particular, 
\begin{equation}
\label{eq:deltaISimpleyhhatSum}
(\delta-I)_{\Q} (\mathfrak{y}(a,b)) \equiv (q-a+1)\cdot \mathfrak{y}(a-1,b) + (b+1)\cdot \mathfrak{y}(a,b+1)
\bmod F_{q-a+b+2}\wedge^kH_{\Q}.
\end{equation}

\medskip

\noindent \textbf{Claim}: 
	For $0\leq p\leq k-q$, $\mathfrak{y}(p,p) $ is in  $\image(\delta-I)_{\Q} \bmod F_{q+1}\wedge^kH_{\Q}$.

\medskip

 The case $p=0$ is exactly the statement that $y\wedge h \in \image (\delta-I)_{\Q} \bmod F_{q+1}\wedge^kH_{\Q}$. We will proceed by downward induction on $p$. Because $\mathfrak{y}(k-q+1, k-q+1) = 0$,  Equation \eqref{eq:deltaISimpleyhhatSum} yields
\begin{equation*}
(\delta-I)_{\Q}\left(\mathfrak{y}(k-q+1, k-q)\right) 
\equiv (2q-k) \cdot \mathfrak{y}(k-q, k-q) \bmod F_{q+1}\wedge^kH_{\Q}
\end{equation*}
and therefore the claim holds when $p=q$, noting that $2q > k$.   Assuming  the Claim is true for $p+1$, we will show that it is true for $p$. Again by Equation \ref{eq:deltaISimpleyhhatSum},
\begin{equation*}
(\delta-I)_{\Q}\left(\mathfrak{y}(p+1,p) \right)  \equiv 
(q-p)\cdot\mathfrak{y}(p,p) +  (p+1) \cdot\mathfrak{y}(p+1,p+1) \bmod F_{q+1}\wedge^kH_{\Q}.
\end{equation*}
By the inductive hypothesis, $\mathfrak{y}(p+1,p+1)$ is in the image of $(\delta-I)_{\Q}\bmod F_{q+1}\wedge^kH_{\Q}$, and therefore so is $(q-p)\mathfrak{y}(p,p)$; by the hypothesis that $q > k/2$, we have $q-p \neq 0$, so $\mathfrak{y}(p,p)$ is in the image of $(\delta-I)_{\Q}\bmod F_{q+1}\wedge^kH_{\Q}$ and we are done.
\end{proof}

\begin{remark}
In the case where $\dim H = 2\dim Y$, the largest possible dimension of $Y$, the bound $q>k/2$ in Lemma \ref{lem:GradedSurjective} is sharp. Set $d=\dim Y$, $e = \dim H$, and $u(q) = \dim(\gr_q^F\wedge^kH)$. Then 
\begin{equation*}
    u(q) = {d \choose q} {e-d\choose k-q}.
\end{equation*}
 When $e=2d$ and $0\leq q\leq k/2$
\begin{equation*}
    \frac{u(q)}{u(q-1)} = \frac{(d-q+1)(k-q+1)}{q(d-k+q)} \geq  \frac{(d-q+1)(q+1)}{q(d-q)}  > 1
\end{equation*}
and therefore $u$ is a strictly increasing function on this interval. This means that  $(\delta-I)_{\Q}$ as in Lemma \ref{lem:GradedSurjective} cannot be surjective.
\end{remark}

We denote by $A_q(\delta)$ and $B_q(\delta)$ the groups:
\begin{align*}
    A_q(\delta) &= \image( H^1(\langle \delta \rangle, F_{q}\wedge^{2q-1}H) \to H^1(\langle \delta \rangle, \wedge^{2q-1}H)), \\
    B_q(\delta) &= \coker(\delta-I:\gr_{q-1}^F\wedge^{2q-1}H \to \gr_{q}^F\wedge^{2q-1}H).
\end{align*}

\begin{proposition}	
\label{prop:AdeltaBdeltaFinite} 
We have isomorphisms of groups
\begin{equation*}
A_q(\delta) \cong \frac{F_q\wedge^{2q-1}H}{((\delta-I)\wedge^{2q-1}H) \cap (F_q\wedge^{2q-1}H)}, \hspace{10pt} B_q(\delta) \cong \frac{F_q\wedge^{2q-1}H}{((\delta-I)F_{q-1}\wedge^{2q-1}H) + (F_{q+1}\wedge^{2q-1}H)}.
\end{equation*}
In particular, $A_q(\delta)$ and $B_q(\delta)$ are finite. 
\end{proposition}
\begin{proof}
It is a standard fact that 
\begin{equation*}
	H^1(\langle \delta \rangle, \wedge^kH) \xrightarrow{}  \wedge^kH/(\delta-I)\wedge^kH \hspace{20pt} [\zeta] \mapsto \zeta(\delta)
\end{equation*}
is an isomorphism. This yields the isomorphism involving $A_q(\delta)$, and the one involving $B_q(\delta)$ is clear. 
Because each $(\delta-I)_{\Q}:\gr_{i-1}^F\wedge^{2q-1}H_{\Q} \to \gr_{i}^F\wedge^{2q-1}H_{\Q}$ is surjective for $i\geq q$ by Lemma \ref{lem:GradedSurjective}, we see that   $F_q\wedge^{2q-1}H_{\Q}$ is contained in $(\delta-I)(F_{q-1}\wedge^{2q-1}H_{\Q})$. Therefore,  $A_q(\delta)$ and $B_q(\delta)$ are  finite.
\end{proof}

\begin{proposition}
\label{prop:intersectDeltaFq}
If $\delta-I:\gr_{q-2}^F\wedge^kH \to \gr_{q-1}^F\wedge^kH$ is injective, then 
\begin{equation*}
(\delta-I)(F_{q-2}\wedge^kH) \cap F_{q} \wedge^kH = (\delta-I)F_{q-1}\wedge^kH.
\end{equation*}
\end{proposition}

\noindent In particular, this means that if $\delta-I:\gr_{i-2}^F\wedge^{2q-1}H \to \gr_{i-1}^F\wedge^{2q-1}H$ is injective for all $i\leq q$, then $A_q(\delta) \cong F_q\wedge^{2q-1}H/(\delta-I)F_{q-1}\wedge^{2q-1}H$, and hence there is a natural surjection $A_q(\delta)\to B_q(\delta)$.

\begin{proof}
Suppose  $y\in (\delta-I)(F_{q-2}\wedge^kH) \cap F_{q}\wedge^kH$
and $y=(\delta-I)x$. By injectivity of $\delta-I:\gr_{q-2}^F\wedge^kH \to \gr_{q-1}^F\wedge^kH$ and the fact that $y\in F_q\wedge^kH$, we see that $x\equiv 0 \mod F_{q-1}\wedge^kH$, i.e., $y\in (\delta-I)F_{q-1}\wedge^kH$.  The other inclusion follows from Lemma \ref{lem:deltaISimpleWedge}. 
\end{proof}

In \S \ref{sec:examples}, we will need to show that certain classes in $H^1(\langle \delta \rangle, \wedge^k H)$ arising from topology are trivial, for which we will need the following explicit computation. 

\begin{proposition}
\label{prop:F1KernelDeltaITrivial}
Any element of $\wedge^k H$ of the form $y \wedge z_1 \wedge \ldots \wedge z_{k-1}$ where $y \in \image(\delta-I)$ and $z_i \in \ker(\delta-I)$ lies in $(\delta-I) \wedge^kH$.
\end{proposition}

\begin{proof}
Choose $x \in H$ such that $(\delta-I)x = y$. Then
\begin{align*}
    (\delta-I)(x\wedge z_1\wedge\ldots \wedge z_{k-1}) =& \delta x \wedge \delta z_1 \wedge \ldots \wedge \delta z_{k-1} - x\wedge z_1\wedge \ldots \wedge z_{k-1} \\
    =& \delta x \wedge z_1 \wedge \ldots \wedge z_{k-1} - x\wedge z_1\wedge \ldots \wedge z_{k-1} \\
    =& y\wedge z_1 \wedge \ldots \wedge z_{k-1}.\qedhere
\end{align*}
\end{proof}

The main application of Proposition \ref{prop:AdeltaBdeltaFinite} will be in the case $k=3$ and we will denote $L=\wedge^3 H$ as before. For this reason, we simply write $A(\delta)$ and $B(\delta)$ for the subgroups $A_2(\delta)$ and $B_2(\delta)$, respectively. These groups are finite by Proposition \ref{prop:AdeltaBdeltaFinite}, and in particular any element of $H^1(\langle \delta \rangle, L)$ which lies in $A(\delta)$ is torsion.

\subsection{The maximal rank case}
\label{subsec:maximalRank}
An important subcase is that where $\dim Y$ is as large as possible, namely $\dim Y= g=\tfrac{1}{2}\dim H$. Because $Y\subset H$ is saturated, there is a subgroup $X\subset H$ such that $H=X\oplus Y$. Let $Q:X \to Y$ denote the restriction of $\delta-I$ to $X$ and  $q_1\leq\cdots\leq q_g$ its invariant factors, with $q_i | q_{i+1}$. Because $Q$ is rationally surjective, each $q_i$ is positive and  
\begin{equation*}
\coker(Q) \cong \prod_{i=1}^g \Z/(q_i). 
\end{equation*}
is a finite group.  Choose bases $\{x_1,\ldots,x_g\}$ of $X$ and $\{y_1,\ldots,y_g\}$ of $Y$ such that $Q(x_i) = q_iy_i$.  To compute $A(\delta)$, we decompose $L = \wedge^3H$ as the direct sum of 
\begin{align*}
    &V_1 = \Span\{x_i\wedge x_j \wedge x_k \, | \, i<j<k\}, \hspace{28pt} V_2 = \Span\{x_i\wedge x_j \wedge y_j \, | \, i\neq j\}, \\
    &V_3 = \Span\{x_i\wedge x_j \wedge y_k \, | \, i<j, k\neq i,j\},  \hspace{10pt} V_4 = \Span\{x_i\wedge y_i \wedge y_j \, | \, i\neq j\},\\
    &V_5 = \Span\{x_i\wedge y_j \wedge y_k \, | \, j<k, i\neq j,k\}, \hspace{10pt} V_6 = \Span\{y_i\wedge y_j \wedge y_k \, | \, i<j<k\},
\end{align*}
Let $A$ denote the matrix of $\delta-I$ with respect to the basis above, and $A_{ij}$ the block whose rows correspond to $V_i$ and columns correspond to $V_j$.

\begin{lemma}
\label{lem:Aij}
The submatrices $A_{ij}$ satisfy the following:
\begin{enumerate}
    \item $A_{13}$, $A_{24}$, and $A_{35}$ are nonsingular,
    \item $\coker(A_{24}) \cong \coker(Q)^{g-1}$,
    \item  $\coker(A_{35}) \cong \prod_{i<j<k}(\Z/(q_i))^2 \times \Z/(2q_jq_k/q_i)$,
    \item $\coker(A_{56}) \cong \prod_{i=1}^{g-2}(\Z/(q_i))^{{g-i\choose 2}}$
    \item $A_{12}, A_{14}, A_{15}, A_{16}, A_{23}, A_{25}, A_{26}, A_{34}, A_{36}, A_{45}, A_{46}$, and $A_{ij}$ for $i\geq j$ are all $0$.
\end{enumerate}
\end{lemma}

\begin{proof}
The matrix $A_{13}$ is nonsingular because its rows each have exactly one nonzero entry. Indeed, the row corresponding to $x_i\wedge x_j\wedge x_k$ has $q_k$ in the column  corresponding to $x_i\wedge x_j\wedge y_k$, and 0 for the remaining entries. 
Next, $A_{24}$ is a square $g(g-1)$ matrix and   $(\delta-1)(x_i \wedge x_j \wedge y_j) = q_i \cdot y_i \wedge  x_j \wedge y_j \mod F_3L$. Therefore $A_{24}$ is nonsingular and its cokernel is of the desired form. 
Now consider $A_{35}$. Given $i<j<k$, set
\begin{align*}
    V_{ijk} &= \Span\{x_i\wedge x_j \wedge y_k,\, x_j\wedge x_k \wedge y_i,\, x_k\wedge x_i \wedge y_j\} \\
    W_{ijk} &=  \Span\{x_k\wedge y_i \wedge y_j,\,   x_i\wedge y_j \wedge y_k,\, x_j\wedge y_k \wedge y_i \}.
\end{align*}
The matrix $A_{35}$ may be arranged into block-diagonal form, where each block has rows indexed by $W_{ijk}$ and columns by $V_{ijk}$. 
With respect to the bases above, this block is
\begin{align*}
    \begin{pmatrix}
    0 & q_j & q_i \\
    q_j & 0 & q_k \\
    q_i & q_k & 0
    \end{pmatrix}
\end{align*}
whose invariant factors are $q_i,q_i,2q_jq_k/q_i$.  Therefore, $\coker A_{35}$ is isomorphic to the product of the $(\Z/(q_i))^2\times \Z/(q_jq_k/q_i)$. Because each block matrix is nonsingular, $A_{35}$ is also nonsingular. This completes the proof of (1), (2), and (3).

Now consider (4). Because $(\delta-I)(x_i\wedge y_j\wedge y_k) = q_i \cdot y_i\wedge y_j\wedge y_k$, each column of $A_{56}$ has exactly one nonzero entry. By only performing column operations, we may form a diagonal matrix $B_{56}$ from  $A_{56}$  so that each diagonal entry is the $\gcd$ of the integers in its corresponding row. Because the nonzero entries of the row in $A_{56}$ corresponding to $y_{i}\wedge y_{j} \wedge y_{k}$  are $q_i,q_j,q_k$, the corresponding diagonal entry in $B_{56}$ is $\gcd(q_i,q_j,q_k) = q_i$. From this, we see that $\coker(A_{56})$ has the desired form.

Finally consider (5). The matrices $A_{ij}$ for $i\geq j$ are all 0 because the matrix of $\delta-I$, with respect to the above decomposition, is strictly lower-triangular. The remaining listed matrices are 0 because  $\delta-I:X\to Y$ is diagonal with respect to the given bases.
\end{proof}

\begin{proposition}
\label{prop:gradedSurjectiveMaximalRank}
The maps $\delta-I:\gr_{q-1}^FL_{\Q} \to \gr_{q}^FL_{\Q}$ are surjective when $q=2,3$ and injective when $q=1,2$. In particular,
\begin{equation*}
A(\delta) \cong \frac{F_2L}{(\delta-I)F_1L}.
\end{equation*}
\end{proposition}

\begin{proof}
The surjectivity claim follows from Lemma \ref{lem:GradedSurjective}. 
The matrix $A_{13}$ is nonsingular by Lemma \ref{lem:Aij}(1), therefore
$\delta-I:\gr_0^FL_{\Q} \to \gr_1^FL_{\Q}$ is injective. Because $\delta-I:\gr_1^FL_{\Q} \to \gr_2^FL_{\Q}$ is  surjective and both spaces have the same dimension (equal to $g{g\choose 2}$), it is also injective.
The last statement follows from Proposition \ref{prop:AdeltaBdeltaFinite} and Proposition \ref{prop:intersectDeltaFq}.
\end{proof}

\begin{proposition}
\label{prop:BdeltaExplicit}
We have an isomorphism
\begin{equation*}
 B(\delta) \cong 
  \coker(Q)^{g-1} \times \prod_{i<j<k}(\Z/(q_i))^{2} \times \Z/(2q_jq_k/q_i).
\end{equation*}
\end{proposition}

\begin{proof}
In terms of the decomposition above,  $\delta-I: \gr_F^1L \to \gr_F^2L$ is given by the block matrix
\begin{equation*}
A = \begin{pmatrix}
    A_{24} & A_{25} \\
    A_{34} & A_{35}
    \end{pmatrix}.
\end{equation*}
 The proposition now follows from Lemma \ref{lem:Aij}.
\end{proof}

\begin{proposition}
\label{prop:AtoB}
We have an isomorphism
\begin{equation*} 
A(\delta) \cong B(\delta) \times \prod_{i=1}^{g-2}(\Z/(q_i))^{{g-i\choose 2}}
\end{equation*}
\end{proposition}

\begin{proof}  
Under the identifications $A(\delta) \cong F_2L/(\delta-I)F_1L$ and $B(\delta) \cong F_2L/((\delta-I)F_1L + F_3L)$ from Propositions \ref{prop:gradedSurjectiveMaximalRank} and \ref{prop:AdeltaBdeltaFinite}, the map $A(\delta) \to B(\delta)$ given by the projection
\begin{equation*}
 \frac{F_2L}{(\delta-I)F_1L} \xrightarrow{} \frac{F_2L}{(\delta-I)F_1L + F_3L}
\end{equation*}
is surjective. Its kernel is isomorphic to $F_3L/(F_3L\cap (\delta-I)F_1L)$. The map $\delta-I:\gr_1^FL\to \gr_2^FL$ is injective by Proposition \ref{prop:gradedSurjectiveMaximalRank}, so $F_3L\cap (\delta-I)F_1L = (\delta-I)F_2L$ by Proposition \ref{prop:intersectDeltaFq}.  Therefore, we have  an exact sequence
\begin{equation*} 
0 \to \frac{F_3L}{(\delta-I)F_2L} \to A(\delta) \to B(\delta) \to 0
\end{equation*}
We claim that this exact sequence splits. The decomposition $F_2L\cong \gr_2^FL\times F_3L$ yields a projection $\pi:F_2L\to F_3L/(\delta-I)F_2L$. Given $a\in (\delta-I)F_1$, express it as $a = a_1+a_2$ with $a_1\in \gr_2^FL$ and $a_2\in F_3$. In fact,  $a_2 \in (\delta-I)F_2$ since $\delta-I:\gr_1^FL \to \gr_2^FL$  is injective. Therefore, $a\in \ker \pi$, and hence $\pi$  induces a splitting of $F_3L/(\delta-I)F_2L \to A(\delta)$. Finally, $F_3L/(\delta-I)F_2L \cong \prod_{i=1}^{g-2}(\Z/(q_i))^{{g-i\choose 2}}$ by Lemma \ref{lem:Aij}(4) and (5).  
\end{proof}

\begin{corollary}
\label{cor:sizeAdeltaBdelta}
When $Y$ has maximal rank,
\begin{equation*}
|A(\delta)| = 2^{{g\choose 3}}\det(Q)^{{g\choose 2}}\prod_{i=1}^{g-2}q_i^{{g-i\choose 2}} \hspace{10pt} \text{and} \hspace{10pt} |B(\delta)| = 2^{{g\choose 3}}\det(Q)^{{g\choose 2}}.
\end{equation*}
\end{corollary}

\begin{proof}
Observe that $\prod_{i<j<k}q_iq_jq_k = \det(Q)^{g-1\choose 2}$ because each invariant factor occurs exactly ${g-1\choose 2}$ times. The formulas for $|B(\delta)|$ and $|A(\delta)|$ now follow from Propositions \ref{prop:BdeltaExplicit} and \ref{prop:AtoB}, respectively.
\end{proof}

\subsection{The symplectic case}
\label{subsec:symplecticCase}
Finally, we consider the case where $H$ is equipped with a symplectic form $\omega\in\wedge^2H$, and $\delta$ is an element of $\Sp(H)$ such that $(\delta-I)^2=0$. We embed $H$ into $L$ via $h\mapsto \omega \wedge h$. Because $\delta$ preserves the form $\omega$, it acts on $L/H$.  We define 
\begin{equation*} 
F_q(L/H) = (F_qL + H)/H. 
\end{equation*}
Recall that $Y$ is the saturation of $\image(\delta-I)$, which is isotropic since $\delta$ is symplectic, and $X$ a subgroup such that $H=X\oplus Y$.
Because $H\subset F_1L$, $F_2L + H = F_2L \oplus X$, and $F_3L\cap H  = 0$, each $F_q(\LH)$ is saturated in $\LH$. In particular, the graded pieces $\gr_q^F(\LH)$ are free. By Lemma \ref{lem:deltaISimpleWedge}, $(\delta-I)$ takes $F_q(\LH)$ to $F_{q+1}(\LH)$, hence induces a map $\gr_{q}^F(\LH) \to \gr_{q+1}^F(\LH)$.  We denote by $\overline{A(\delta)}$ and $\overline{B(\delta)}$ the groups 
\begin{align*}
    \overline{A(\delta)} &= \image( H^1(\langle \delta \rangle, F_{2}(\LH)) \to H^1(\langle \delta \rangle, \LH)), \\
    \overline{B(\delta)} &= \coker(\delta-I:\gr_{1}^F(\LH) \to \gr_{2}^F(\LH)).
\end{align*}

As we will show in the next section, the Ceresa class $\nu(f)$ of a positive multitwist $f$ on a closed surface, with symplectic representation $\delta_f$, lives in $\overline{A(\delta_f)}$, provided $Y$ has maximal rank. In this subsection, we will show that $\overline{A(\delta)}$ is finite, from which we conclude that the Ceresa class is torsion. When $Y$ has maximal rank, $\overline{A(\delta)}$ naturally surjects onto $\overline{B(\delta)}$. It is much easier to compute the image of $\nu(f)$ in  $\overline{B(\delta_{f})}$, see Equation \eqref{eq:nuBdelta}. We use this in \S\ref{sec:examples} to determine non-triviality of  $\nu(f)$ in some examples.

The following three propositions, and their proofs, are similar to Propositions \ref{prop:AdeltaBdeltaFinite},  \ref{prop:gradedSurjectiveMaximalRank}, and  \ref{prop:AtoB}.

\begin{proposition}
\label{prop:AbarBbar}
    We have isomorphisms
\begin{equation*}
\overline{A(\delta)} \cong \frac{F_2L + H}{((\delta-I)L + H) \cap (F_2L + H)}, \hspace{10pt} \overline{B(\delta)} \cong \frac{F_2L+H}{(\delta-I)(F_{1}L) + F_{3}L + H}.
\end{equation*}
In particular, $\overline{A(\delta)}$ and $\overline{B(\delta)}$ are finite. 
\end{proposition}

\begin{proposition}
\label{prop:gradedSurjectiveModH}
The map $(\delta-I)_{\Q}:\gr_{q-1}^F(\LH)_{\Q} \to \gr_{q}^F(\LH)_{\Q}$ is surjective when $q=2,3$. When $Y$ has maximal rank, this map is injective for $q=1,2$ and
\begin{equation*}
\overline{A(\delta)} \cong \frac{F_2L + H}{(\delta-I)F_1L + H}.
\end{equation*}
\end{proposition}

\begin{proposition}
\label{prop:AbarToBbar}
If $Y$ has maximal rank, then
\begin{equation*}
\overline{A(\delta)} \cong \overline{B(\delta)} \times \prod_{i=1}^{g-2}(\Z/(q_i))^{{g-i\choose 2}}.
\end{equation*}
\end{proposition}

The next two propositions compare $A(\delta)$ and $B(\delta)$ from the previous subsection to their counterparts $\overline{A(\delta)}$ and $\overline{B(\delta)}$. 

\begin{proposition}
\label{prop:BtoBbar}
If $Y$ has maximal rank, then we have an exact sequence
\begin{equation*}
0 \to \coker(Q) \to B(\delta) \to \overline{B(\delta)} \to 0.
\end{equation*}
\end{proposition}

\begin{proof}
Consider the following commutative diagram
\begin{equation*}
\begin{tikzcd}
& 0 \arrow[r] & H \arrow[r, "\wedge \omega"] \arrow[d, "\delta-I"]& F_1L \arrow[r] \arrow[d, "\delta-I"] & F_1(\LH) \arrow[r] \arrow[d, "\delta-I"] &0 \\
& 0 \arrow[r] & Y \arrow[r, "\wedge \omega"] & F_2L \arrow[r]  & F_2(\LH) \arrow[r]  &0 
\end{tikzcd}
\end{equation*}
whose rows are exact. The map $\delta-I:F_1(\LH) \to F_2(\LH)$ is injective because it becomes an isomorphism after tensoring with $\Q$ by Proposition \ref{prop:gradedSurjectiveModH}.  We now get the desired exact sequence by the snake lemma.
\end{proof}

\begin{proposition}
\label{prop:AtoAbar}
If $Y$ has maximal rank, then we have an exact sequence
\begin{equation*}
0 \to \coker(Q) \to A(\delta) \to \overline{A(\delta)} \to 0.
\end{equation*}
\end{proposition}

\begin{proof}
First, observe that $A(\delta)\to \overline{A(\delta)}$ induced by
\begin{equation*}
\frac{F_2L}{(\delta-I)F_1L} \to \frac{F_2L + H}{(\delta-I)F_1L + H}
\end{equation*}
is surjective. Let $K$ denote its kernel. Then we have the following commutative diagram 
\begin{equation*}
\begin{tikzcd}
& 0 \arrow[r] & K \arrow[r, ] \arrow[d]& A(\delta) \arrow[r] \arrow[d] & \overline{A(\delta)} \arrow[r] \arrow[d] &0 \\
& 0 \arrow[r] & \coker(Q) \arrow[r] & B(\delta) \arrow[r]  & \overline{B(\delta)} \arrow[r]  &0 
\end{tikzcd}
\end{equation*}
whose rows are exact. The vertical map on the left is an isomorphism by Propositions \ref{prop:AtoB}, \ref{prop:AbarToBbar}, and the snake lemma. 
\end{proof}

\begin{corollary}
\label{cor:sizeAbarBbar}
When $Y$ has maximal rank,
\begin{equation*}
|\overline{A(\delta)}| = 2^{{g\choose 3}}\det(Q)^{{g\choose 2}-1}\prod_{i=1}^{g-2}q_i^{{g-i\choose 2}} \hspace{10pt} \text{and} \hspace{10pt} |\overline{B(\delta)}| = 2^{{g\choose 3}}\det(Q)^{{g\choose 2}-1}. 
\end{equation*}
\end{corollary}

\section{The Ceresa class of a multitwist}

\label{sec:ceresaClassMultitwist}

\subsection{Dehn Twists and Multitwists}
\label{subsec:dehnTwistMultitwist}
In this subsection, we recall some basic facts about Dehn twists. We refer the reader to \cite{FarbMargalit} for a more detailed treatment.

\begin{lemma}
\label{lem:propertiesDehnTwist}
Let $f \in \Gamma_g$ (resp. $\Gamma_g^1$) and $a$ the isotopy class of a simple closed curve. We have
\begin{enumerate}
    \item  $T_{f(a)} = f T_a f^{-1}$, in particular, $[T_a,f] = T_a T_{f(a)}^{-1}$, and
    \item $f$ commutes with $T_a$ if and only if $f(a) = a$; in particular $T_aT_b = T_bT_a$ if and only if the geometric intersection number $i(a,b)=0$.
\end{enumerate}
\end{lemma}
\begin{proof}
See \cite[Facts~3.7, 3.8]{FarbMargalit}.
\end{proof}
As before, set $H = H_1(\Sigma_g,\Z) = H_1(\Sigma_g^1,\Z)$. We write $[a]$ for the homology class of an isotopy class $a$ of a simple closed curve, and $\hat{i}$ for the algebraic intersection number on $H$. 
The induced map $(T_{a})_* \in \Sp(H)$  only depends on the homology class $[a]$, and
\begin{equation}
\label{eq:symplecticRepDehnTwist}
    (T_a)_* ([b]) = [b] + \hat{i}([a],[b])[a],
\end{equation}
see \cite[Proposition 6.3]{FarbMargalit}.

Let $\Lambda$ be a collection of isotopy classes of pairwise non-intersecting  essential simple closed curves in $\Sigma_g^1$ (resp. $\Sigma_g$).  
Define 
\begin{equation*}
    \pB = \langle T_{\ell} \, :\, \ell \in \Lambda \rangle < \Gamma_g^1 \; \text{ or } \; \Gamma_g, \hspace{10pt} \text{and} \hspace{10pt}  Y = \langle [\ell] \, : \, \ell\in \Lambda \rangle < H.
\end{equation*}
The group $\pB$ is free and abelian by Lemma \ref{lem:propertiesDehnTwist}(2), and $Y$ is saturated in $H$ since any $d=\rank Y$ collection of simple closed curves $\lambda_1,\ldots,\lambda_d$ whose homology classes are linearly independent extends to an integral basis of $H$. Given $f\in \pB$, we write $\delta_{f} \in \Sp(H)$ for the image of $f$ under the symplectic representation (we could also call this $f_*$,  but use $\delta_f$ to better match the notation of \S\ref{sec-algebra-wedge-H}). An element of the form 
\begin{equation*}
    f = \prod_{\ell\in \Lambda} T_{\ell}^{c_{\ell}} \in \pB
\end{equation*}
with $c_{\ell}>0$ for all $\ell$ is called a \textit{positive multitwist} supported on $\Lambda$.


\begin{proposition}
\label{prop:YSpan}
For any positive multitwist $f$ supported on $\Lambda$, we have that $Y$ is the saturation of $\image (\delta_{f}-I)$. 
\end{proposition}

\begin{proof}
We temporarily denote by $Y'$ the saturation of $\image (\delta_{f}-I)$ in $H$.  Applying  Equation \eqref{eq:symplecticRepDehnTwist} to \textit{any} $f = \prod T_{\ell}^{c_{\ell}}$ yields
\begin{equation}
\label{eq:fMinusIOnH}
(\delta_{f}-I)(h) = \sum_{\ell \in \Lambda} c_{\ell} \, \hat{i}([\ell],h)[\ell],
\end{equation}
and therefore $(\delta_{f}-I)(H) \leq Y$. Applying this to a positive multitwist $f$, we have $Y'\leq Y$. Because of this and the fact that $Y$ and $Y'$ are saturated subgroups of $H$, the equality $Y_{\Q} = Y_{\Q}'$ will imply $Y=Y'$.  
Denote by $(Y_{\Q})^{\omega} \leq H_{\mathbb{Q}}$ the symplectic complement to $Y_{\Q}$, i.e.,
\begin{equation*}
    (Y_{\Q})^{\omega} = \{h\in H_{\Q} \, | \, \hat{i}( y,h )= 0 \text{ for all } y\in Y_{\Q}\}
\end{equation*}
and set $W_{\Q} = H_{\Q}/(Y_{\Q})^{\omega}$. Because $Y_{\Q}$ is an isotropic subspace of $H_{\Q}$, we have  $\dim W_{\mathbb{Q}} = \dim Y_{\Q}$.  Since $(Y_{\Q})^{\omega}$ lies in the kernel of $(\delta_{f}-I)_{\Q}$ by Equation \eqref{eq:fMinusIOnH}, the following 
\begin{equation*}
    (\alpha,\alpha') = \hat{i}( \alpha, (\delta_{f}-I)\alpha')
\end{equation*}
defines a bilinear form on $W_{\Q}$. We claim that this is actually an inner product. Indeed, we have 
\begin{equation*}
    \hat{i}( \alpha,(\delta_{f}-I)\alpha') = \sum_{\ell\in \Lambda} c_{\ell}\, \hat{i}([\ell], \alpha) \hat{i}([\ell], \alpha'),
\end{equation*}
and hence this bilinear form is symmetric and positive semidefinite.  If $\alpha$ is a nonzero element of $W_{\Q}$, then there is some  $\ell \in \Lambda$ such that $\hat{i}(  [\ell],\alpha) \neq 0$. By the above equation, we see that $(\alpha,\alpha) > 0$, which establishes the positive definiteness. We conclude that $(\delta_{f}-I)_{\Q}: W_{\Q} \to Y_{\Q}'$ is an injective linear map of vector spaces that have the same dimension, and hence is also surjective. This proves  $ Y_{\Q} = Y'_{\Q}$, and the proposition follows. 
\end{proof} 

The following proposition will be used for ``computing'' the Ceresa class of a Lagrangian collection of curves, in the sense of \S \ref{sec:LagrangianTorsion}. 





\begin{proposition}
\label{prop:splitDelta}
Take elements $f_1,\ldots, f_M \in \pB$, let $f$ denote their product, and let $\tau$ be a hyperelliptic quasi-involution. Then 
\begin{equation*}
  J([f,\tau]) = \sum_{i=1}^M (g_i)_* \cdot J([f_i,\tau]) 
\end{equation*}
for some $g_1,\ldots,g_M \in \pB$. 
\end{proposition}

\begin{proof}
We proceed by induction on $M$. When $M=1$ the Lemma is clear. Suppose that the Lemma is true for $M-1$. Then by Proposition \ref{prop:muIndependentTau} we have 
\begin{align*}
    J([f,\tau]) =J([f_1,\tau])+(f_1)_{*} \cdot J([f_2\cdots f_M,\tau])  =J([f_1,\tau]) + (f_1)_{*}  \sum_{i=2}^M  (g_{i})_{*}\cdot J([f_i,\tau]).
\end{align*}
Since each $f_i$ lies in $\pB$, we have placed $J([f,\tau])$ in the desired form.
\end{proof}

\subsection{Handlebodies and the Luft-Torelli group}
\label{sec:LuftTorelli}
Let $V$ be a handlebody with boundary $\Sigma_g$. Let $D\subset \Sigma_g$ be a small open disc so that $\Sigma_g^1 = \Sigma_{g} \setminus D$.  The \textit{handlebody group} $\pH_{g}^{1}(V)$  is the subgroup of $\Gamma_g^1$  consisting of mapping classes that are restrictions of homeomorphisms of $V$. Denote by $\pL_g^1(V)$  the kernel of the homomorphism 
\begin{equation*}
    \pH_{g}^1(V) \to \Aut \pi_1(V). 
\end{equation*}
The \textit{Luft-Torelli group} of $\Sigma_g^{1}$   is
\begin{equation*}
    \LT_g^1(V) = \pL_g^1(V) \cap \cI_g^1. 
\end{equation*}
A \textit{meridian} is a nontrivial isotopy class of a simple closed curve in $\Sigma_g^1$ that bounds a properly embedded disc in $V$. Note that $T_{\ell}$ lies in $\pL^1_g(V)$ if $\ell$ is a meridian. A \textit{contractible bounding pair} is a bounding pair $(\gamma,\gamma')$ such that $\gamma$ and $\gamma'$ are meridians, and a \textit{contractible bounding pair map} is the product of Dehn twists $T_{\gamma}T_{\gamma'}^{-1}$ where $(\gamma,\gamma')$ is a contractible bounding pair. The following is \cite[Theorem~9]{Pitsch}.
\begin{theorem}
\label{thm:LuftTorelli}
For $g\geq 3$, the Luft-Torelli group $\LT_g^1(V)$ is generated by contractible bounding pair maps.
\end{theorem}

Given a handlebody $V$, the kernel $Y_{V}$ of the map
\begin{equation*}
    H_{1}(\Sigma_g^1, \Z) \to H_1(V, \Z)
\end{equation*}
induced by the inclusion $\Sigma_g^1 \hookrightarrow V$ is a Lagrangian subspace of $H$. Define a filtration of $L$ similar to the one in Equation \eqref{eq:Ffiltration}:
\begin{equation*}
    F_q^{V}L = (\wedge^{q} Y_V) \wedge (\wedge^{3-q} H).
\end{equation*}

\begin{proposition}
\label{prop:LuftTorelliF2}
If $f$ is a contractible bounding pair map, then $J(f) \in F_2^V(L)$. In particular, if $g\geq 3$ and
if $f \in \LT_g^1(V)$, then $J(f) \in F_2^V(L)$.
\end{proposition}

\begin{proof}
By Theorem \ref{thm:LuftTorelli}, it suffices to prove the first statement. Let $(\gamma,\gamma')$ be a contractible bounding pair; we may assume that $\gamma$ and $\gamma'$ are as in Figure \ref{fig:basisJohnsonHom}. Using the basis in this figure, we compute
\begin{equation*}
    J(T_{\gamma}T_{\gamma'}^{-1}) = \sum_{i=1}^{d-1} \alpha_i \wedge \beta_i \wedge [\gamma].
\end{equation*}
Since $\beta_1,\ldots,\beta_{d-1},[\gamma]$ lie in $Y$, we see that $J(T_{\gamma}T_{\gamma'}^{-1})$ lies in $F_2^{V}L$, as required. 
\end{proof}

\begin{figure}
    \centering
    \includegraphics[width=0.8\textwidth]{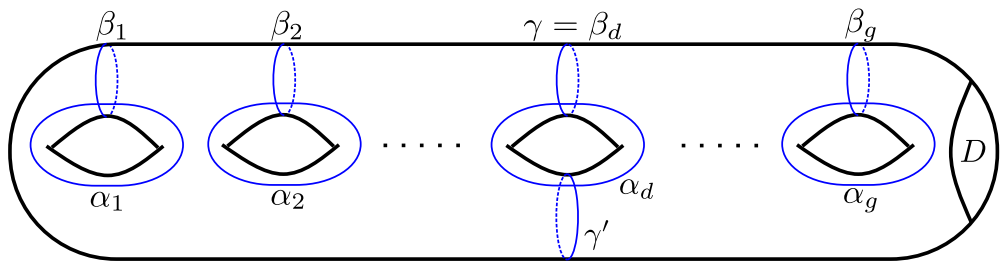}
    \caption{A basis of $\Sigma_g^1$ to compute $J(T_{\gamma}T_{\gamma'}^{-1})$; here the $\beta_1,\ldots,\beta_g$ are meridians of a handlebody }
    \label{fig:basisJohnsonHom}
\end{figure}

\subsection{The Lagrangian case}
\label{sec:LagrangianTorsion}
In this section and the next, we will prove our main result, that the Ceresa class is torsion. We begin by focusing on the case $\Sigma_g^1$; the closed surface case follows readily from this, as we explain in \S \ref{sec:ceresaClosedSurface}.

We say that a collection of nonintersecting simple closed curves $\Lambda$ is \textit{Lagrangian} if the rank of $Y$ is half the rank of $H$, the largest possible rank of $Y$. 
Given any Lagrangian $\Lambda$ on $\Sigma_g^1$ or $\Sigma_g$, there is a handlebody $V$ of $\Sigma_g$ such that each curve in $\Lambda$ is a meridian of $V$; see the proof of \cite[Lemma~5.7]{Hensel}. In this case, the subgroup $Y<H$ has three characterizations:
\begin{enumerate}
    \item $Y$ is the integral span of the curves $[\ell]$ for $\ell \in \Lambda$ (by definition),
    \item $Y$ is the saturation of $\image(\delta_{f}-I)$ for any positive multitwist $f$ supported on $\Lambda$ (by Proposition \ref{prop:YSpan}),
    \item $Y$ is the kernel $Y_V$ of the homomorphism $H_1(\Sigma_g,\Z) \to H_{1}(V,\Z)$. 
\end{enumerate}
Therefore, the filtrations $F_q^V L$ and $F_q L$ agree.

\begin{theorem}
\label{thm:LagrangianTorsion}
Suppose $\Lambda$ is a collection of pairwise disjoint simple closed curves in $\Sigma_g^1$.  Let $f = \prod_{\ell\in \Lambda} T_{\ell}^{c_{\ell}}$ be a positive multitwist, and let V be a handlebody in which each curve in $\Lambda$ is a meridian. Choose a hyperelliptic quasi-involution $\tau$ that lies in $\pH_g^1(V)$. Then $J([f,\tau])$ lies in $F_2^V(L)$.
 In particular, if $\Lambda$ is Lagrangian and $f$ is a positive multitwist supported on $\Lambda$, then  $\mu(f) \in A(\delta_{f})$ and  $\mu(f)$ is torsion.  
\end{theorem}

\begin{proof}
Suppose $g\geq 3$.  The commutator $[f,\tau]$ lies in $\pL_{g}^1(V)$ since each $T_{\ell}$ lies in $\pL_g^1(V)$ and $\pL_g^1(V)$ is a normal subgroup of $\pH_g^1(V)$. The commutator also lies in $\cI_g^1$ since $\tau$ maps to the center of $\Sp(H)$. So $[f,\tau] \in \LT_g^1(V)$ and therefore $J([f,\tau]) \in F_2^V(L)$ by Proposition \ref{prop:LuftTorelliF2}.  

Next, suppose $g=2$. There are $15$  arrangements $\Lambda$ on $\Sigma_2^1$; it suffices to consider the 3 maximal arrangements as illustrated in Figure \ref{fig:genus2n1}. For each surface, regard the ``inside'' as the handlebody $V$ and $\tau$ is rotation by $180^{\circ}$ about the axis horizontal to the page. For the left or middle case, label the isotopy classes in $\Lambda$ in any order by $\ell_1, \ell_2, \ell_3, \ell_4$ and let $c_i=c_{\ell_i}$. The eight isotopy classes $\{\ell_i,\tau(\ell_i) \, : \, i=1,2,3,4\}$ pairwise have geometric intersection number $0$. Thus the corresponding collection of Dehn twists commute, so
\begin{equation*}
    [f,\tau] = \prod_{i=1}^4 (T_{\ell_i} T_{\tau(\ell_i)}^{-1})^{c_i}.
\end{equation*}
Each $T_{\ell_i} T_{\tau(\ell_i)}^{-1}$ is either the identity or a contractible bounding pair, whence  $J([f,\tau]) \in F_2^V(L)$ by Proposition \ref{prop:LuftTorelliF2}.  

Finally, consider the arrangement on the right in Figure \ref{fig:genus2n1} and label the isotopy classes left to right by $\ell_1,\ell_2,\ell_3,\ell_4$; the curves $\ell_2$ and $\ell_3$ are separating. Since $T_{\ell_1}T_{\tau(\ell_1)}^{-1} = T_{\ell_2}T_{\tau(\ell_2)}^{-1} = \id$ and $\ell_3, \tau(\ell_3)$ are separating curves, we have
\begin{equation*}
    J([f,\tau]) = c_{4} J(T_{\ell_4}T_{\tau(\ell_4)}^{-1})
\end{equation*}
which lies in $F_2(L)$ by Proposition \ref{prop:LuftTorelliF2} since $(\ell_4, \tau(\ell_{4}))$ is a contractible bounding pair.

The last statement of the theorem follows from the fact that $F_2^V(L) = F_2(L)$ in the Lagrangian case and the finiteness of $A(\delta_{f})$.   
\end{proof}

\begin{figure}[tbh]
	\centering
	\includegraphics[width=\textwidth]{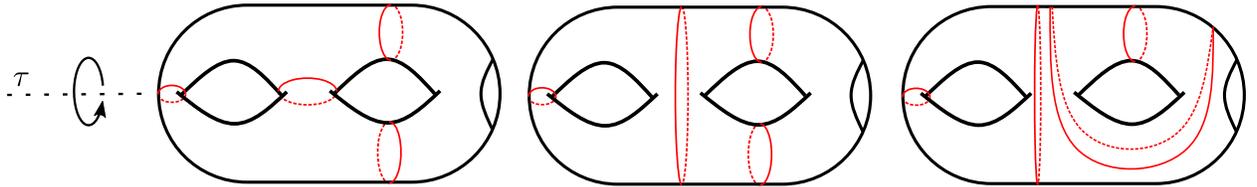}
	\caption{Maximal arrangements of non-intersecting curves on $\Sigma_2^1$} \label{fig:genus2n1}
\end{figure}

Suppose now that $\Lambda$ is Lagrangian and $f$ is a positive multitwist supported on $\Lambda$. Because $\mu(f) \in A(\delta_{f})$, we can consider the image of this class in $B(\delta_{f})$, which we denote by $w_{f}$. The reason to consider $w_{f}$ is that it admits a relatively simple formula. 

\begin{proposition}
\label{prop:muBdelta}
If $\Lambda$ be a Lagrangian arrangement of curves on $\Sigma_g$ and $f = \prod T_{\ell}^{c_{\ell}}$ is a positive multitwist supported on $\Lambda$,  then
\begin{equation*}
 w_{f}  =  \sum_{\ell\in \Lambda} c_{\ell} \cdot J([T_{\ell},\tau]).
\end{equation*}
\end{proposition}

\begin{proof}
By Proposition  \ref{prop:splitDelta}, 
\begin{equation*}
    J([f,\tau]) = \sum_{\ell \in \Lambda} c_{\ell} \cdot (g_{\ell})_* \cdot J([T_{\ell},\tau]) \hspace{20pt} \text{in } \hspace{5pt} B(\delta_{f})
\end{equation*}
where each $g_{\ell}\in \pB$. The proposition now follows from the fact that  $\pB$ acts trivially on $B(\delta_{f})$.
\end{proof}

\subsection{The non-Lagrangian case}

As we will see in Example \ref{example:thetaGraph}, when the collection of homology classes of the curves in $\Lambda$ do not span a Lagrangian subspace of $H$, then the Ceresa class of a positive multitwist $f$ need not live in $A(\delta_{f})$. Nevertheless, this class is still torsion.

\begin{theorem}
\label{thm:nonLagrangianTorsion}
Suppose $\Lambda$ is any collection of pairwise disjoint simple closed curves in $\Sigma_g^1$, and $f = \prod_{\ell\in \Lambda} T_{\ell}^{c_{\ell}}$ is a positive multitwist. Then $\mu(f)$ is torsion. 
\end{theorem}

The proof will occupy this whole subsection. As usual, let $Y = \langle [\ell] \, :\, \ell\in \Lambda \rangle$, and  $d = \dim Y$.  Choose a subset of loops $\ell_1, \ldots, \ell_d$ in $\Lambda$ which are linearly independent (and thus rationally span $Y$.) Fix a collection of isotopy classes of simple closed curves $a_1,\ldots,a_g, b_1,\ldots,b_g$ on $\Sigma_g^1$ such that
\begin{enumerate}
    \item the classes $a_i$ and $b_j$ for $i,j=1,\ldots,g$ are pairwise nonintersecting except $i(a_i,b_i)=1$;
    \item the classes $a_i$ and $b_j$ for $i,j=1,\ldots,g-d$ do not intersect the classes in $\Lambda$;
    \item the homology classes $\alpha_i = [a_i]$ and $\beta_j = [b_j]$ form a symplectic basis for $H$;
    \item $b_{g+1-j} = \ell_j$ for $j=1,\ldots, d$.
\end{enumerate}
Let us explain why such a collection of curves exists. Let $\bG = (G,w)$ be the weighted dual graph of $\Lambda$. Recall that the sum of the weights $w(v)$ over all vertices $v$ is $g-d$. For each vertex $v$ with $w(v)>0$, choose a subsurface $S\cong \Sigma_{w(v)}^1$ of $\Sigma_v$ so that $\partial \Sigma_v$ lies in $\Sigma_v \setminus S$. Now define $w(v)$ of the $a_i$ and $b_i$ for $i\leq g-d$ by the basis in Figure \ref{fig:basisJohnsonHom}. We have already specified $b_{g-d+1},\ldots,b_g \in \Lambda$. Finally, we may choose the remaining $a_i$ to complete the symplectic basis by  \cite[Lemma~A.3]{PutmanCutAndPaste}.


\begin{figure}
    \centering
    \includegraphics[height=5cm]{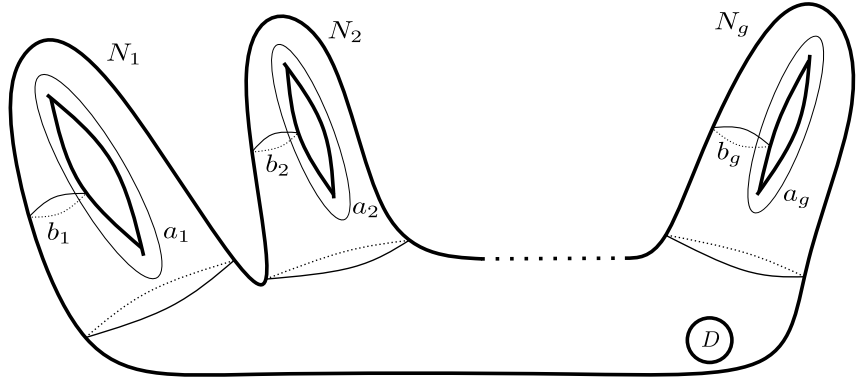}
    \caption{A handlebody such that each $b_i$ is a meridian}
    \label{fig:regularNbd}
\end{figure}

Choose closed regular neighborhoods $N_i$ of $a_i \cup b_i$; each $N_i$ is homeomorphic to a torus with one boundary component, and $\partial N_i$ is a separating curve. Choose $N_i$ small enough so that, for each $i\leq g-d$, the boundary $\partial N_i$ does not intersect the curves in $\Lambda$.
The mapping class group $\Gamma_1^1$ is isomorphic to $\SL_2(\Z)$ and the mapping class $\tau_i$ corresponding to $-I\in \SL_2(\Z)$ takes $a_i$ to $a_i$ and $b_i$ to $b_i$, reversing their orientations. We claim that if $M$ is a handlebody of $N_i$, then $\tau_i \in \pH_1^1(M)$. Since any two handlebody subgroups are conjugate to each other \cite[\S~3]{Hensel} and $-I$ lies in the center of $\SL_2(\Z)$, it suffices to show that $\tau_i$ lies in $\pH_1^1(M)$ for just one $M$. Take $\Sigma_1^{1}$ to be the surface in \ref{fig:basisJohnsonHom} (for $g=1$) and let $M$ be the handlebody on the inside. A representative of the mapping class $\tau_i$ is given by a rotation of $180^{\circ}$ horizontal to the page, and applying a suitable isotopy so that it is the identity on $D$. Extending by the identity defines $\tau_i$ on $\Sigma_{g}^{1}$, and the product  $\tau = \tau_1 \cdots \tau_g$ is a hyperelliptic quasi-involution on $\Sigma_{g}^{1}$.

For $i=1,\ldots,g-d$, let $A_i$ and $B_i$ be handlebodies for $N_i$ so that $a_i$ is a meridian in $A_i$ and $b_i$ is a meridian in $B_i$. The surface $S = \Sigma_g^1 \setminus \cup_{i=1}^{g-d} N_i^{\circ}$ is homeomorphic to $\Sigma_{d}^{g-d+1}$ (recall that $N_i^{\circ}$ denotes the relative interior of $N_i$).  Let $V$ be a handlebody for the surface obtained by capping off the boundary components of $S$ so that $\ell$ is a meridian of $V$ for each $\ell \in \Lambda$. For any 2-part partition $I,J$ of $\{1,\ldots,g-d\}$, let $V_{IJ}$ be the handlebody of $\Sigma_g^1$ obtained by attaching $A_i$ and $B_j$ to $V$ for $i\in I$ and $j\in J$.  By the previous paragraph, we have that $\tau \in \pH_g^1(V_{IJ})$. 



Let
\begin{equation*}
    \pF = \bigcap_{I,J} F_2^{V_{IJ}} (L)
\end{equation*}
where the intersection is taken over all 2-part partitions $I,J$ of $\{1,\ldots,g-d\}$.  We note that in the case already treated, where $\Lambda$ is Lagrangian, $g=d$, so $I$ and $J$ are both empty and there is only a single choice for $V_{IJ}$, namely the handlebody $V$ of the previous section.

\begin{lemma}
\label{lem:tauInVIJ}
The class
$J([f,\tau])$ lies in $\pF$. 
\end{lemma}

\begin{proof}
By Proposition \ref{prop:LuftTorelliF2}, it suffices to show that $[f,\tau] \in \LT_g^1(V_{IJ})$ for all 2-part partitions $I,J$ of $\{1,\ldots,g-d\}$. 
The multitwist $f$ lies in each $\pL_g^1(V_{IJ})$, and $\tau$ lies in each $\pH_g^1(V_{IJ})$. Since $\pL_g^1(V_{IJ})$ is a normal subgroup of $\pH_g^1(V_{IJ})$ and the symplectic representation of $\tau$ is in the center of $\Sp(H)$, we have that $[f,\tau]\in \LT_g^1(V_{IJ})$, as required.
\end{proof}

\begin{lemma}
\label{lem:pFFinite}
The image of $\pF$ in  $L/(\delta_{f}-I)L$, which can be expressed as
\begin{equation}
\label{eq:pFInCohomology}
    \frac{\pF + (\delta_{f}-I)L}{(\delta_{f}-I)L} \cong \frac{\pF}{(\delta_{f}-I)L \cap \pF} \; ,
\end{equation}
is a finite group.
\end{lemma}

\begin{proof}
To prove that the group in \eqref{eq:pFInCohomology} is finite, it suffices to show that 
\begin{equation*}
    (\delta_{f}-I)_{\Q}(L_{\Q}) \cap \pF = \pF.
\end{equation*}
The basis $\alpha_i,\beta_j$ of $H$ induces coordinates on $L\cong \Z^{\binom{2g}{3}}$, and each $F_2^{V_{IJ}}(L)$ is a coordinate subspace of $L$ in this basis. Therefore their intersection $\pF$ is also a coordinate subspace. It is generated by 
\begin{enumerate}
    \item $\abb{i}{j}{k}$ for $j,k > g-d$;
    \item $\bbb{i}{j}{k}$ for $j,k > g-d$;
    \item $\abb{i}{i}{k}$ for $i\leq g-d$ and $k > g-d$. 
\end{enumerate}
The simple wedges in (1) and (2) are already in $F_2L$, so they are in the image of $(\delta_{f}-I)_{\Q}$ by Proposition \ref{prop:AdeltaBdeltaFinite}. For (3), let $x\in H_{\Q}$ satisfy $(\delta_{f}-I)_{\Q}(x) = \beta_k$. Since $(\delta_{f}-I)(\alpha_i) = 0$ and $(\delta_{f}-I)(\beta_i) = 0$ for $i\leq g-d$, we have  $(\delta_{f}-I)_{\Q}(\alpha_i \wedge \beta_i \wedge x) = \abb{i}{i}{k}$. So $(\delta_{f}-I)_{\Q}(L_{\Q}) \cap \pF = \pF$, as required. 
\end{proof}

\begin{proof}[Proof of Theorem \ref{thm:nonLagrangianTorsion}]
The class $[f,\tau]$ lies in $\pF$ by Lemma \ref{lem:tauInVIJ} and so $J([f,\tau])$ lies in the group in \eqref{eq:pFInCohomology}. This group is finite by Lemma \ref{lem:pFFinite}, and therefore $\mu(f)$ is torsion. 
\end{proof}

\subsection{The closed surface case}
\label{sec:ceresaClosedSurface}

Let $\Lambda$ be a collection of pairwise nonintersecting nontrivial isotopy classes of simple closed curves on $\Sigma_g$, and let $f = \prod_{\ell\in \Lambda} T_{\ell}^{c_{\ell}}$ be a positive multitwist. The natural map $\Gamma_g^1 \to \Gamma_g$ is surjective, so we may view $f$ as a positive multitwist on $\Sigma_g^1$. Also, the natural map $H^1(\langle \delta_{f}\rangle , L) \to H^1(\langle \delta_{f}\rangle , L/H)$ takes $\mu(f)$ to $\nu(f)$, so the torsionness of $\nu(f)$ follows from the torsionness of $\mu(f)$. This completes the proof that the Ceresa class of a multitwist on a closed surface is torsion.

The map $H^1(\langle \delta_{f}\rangle, L) \to H^1(\langle \delta_{f}\rangle, L/H)$ also induces a surjection $A(\delta_{f}) \to \overline{A(\delta_{f})}$ as studied in Proposition \ref{prop:AtoAbar}, so when $\Lambda$ is Lagrangian, the class $\nu(f)$ lies in $\overline{A(\delta_{f})}$. The image of $w_{f}$ under $B(\delta_{f}) \to \overline{B(\delta_{f})}$ is the image of $\nu(f)$ under $\overline{A(\delta_{f})} \to \overline{B(\delta_{f})}$, which we denote by $v_{f}$. By Proposition \ref{prop:muBdelta}, we have 
\begin{equation}
\label{eq:nuBdelta}
    v_{f} = \sum_{\ell\in \Lambda} c_{\ell} \cdot J([T_{\ell},\tau]).
\end{equation}

\section{Examples}
\label{sec:examples}

\begin{figure}[tbh]
	\centering
	\includegraphics[height=30mm]{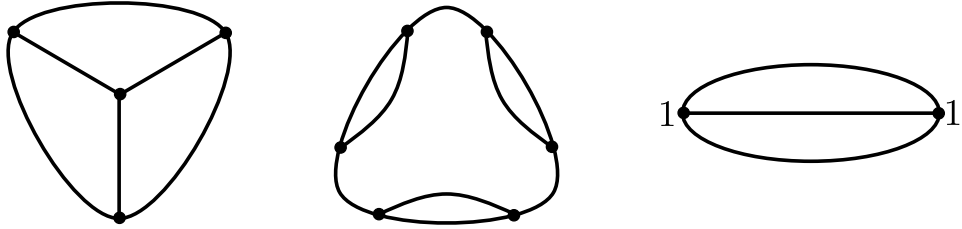}
	\caption{From left to right, these are the graphs $K_4$, $TL_3$. The vertices of the theta graph each have weight 1} \label{fig:graphs}
\end{figure}

\noindent In this section, we will compute the Ceresa class for the multitwist $T_{\Gamma}$ for tropical curves whose underlying vertex-weighted graph is displayed in Figure \ref{fig:graphs}. 
In the first two examples, we will consider tropical curves $\Gamma$ whose vertex-weights are all 0 (see the left and middle graphs in Figure \ref{fig:graphs}). Let us describe our strategy for determining Ceresa nontriviality in this case. 

Let $\Gamma$ be a genus $g\geq 3$ tropical curve with integral edge-lengths $c_1,\ldots,c_N$ and whose vertex weights are all 0.  Let $\ell_{1},\ldots,\ell_{N}$ be the collection of loops on $\Sigma_g$ corresponding to $\Gamma$ as in \S\ref{subsec:tropicalCurves}.  As discussed in \S\ref{sec:ceresaClosedSurface}, the class $\nu(T_{\Gamma})$ lies in $\overline{A(\delta_{\Gamma})}$, and its image in $\overline{B(\delta_{\Gamma})}$, denoted by $v_{\Gamma}$, is described by Equation \eqref{eq:nuBdelta}. So  
 to compute $v_{\Gamma}$, it suffices to compute each $J([T_{\ell_i},\tau])$  for some hyperelliptic involution $\tau$. While the choice of $\tau$ does not matter, it is essential that we use the same $\tau$ to compute each  $J([T_{\ell_i},\tau])$. Now, observe  that $[T_{\ell_i}, \tau] = T_{\ell_i}T_{\tau(\ell_i)}^{-1}$ and $\ell_i$ is homologous to $\tau(\ell_i)$. Three things may happen.
\begin{itemize}
    \item[-] If $\ell_i$ is a separating curve, then so is $\tau(\ell_i)$, and hence $J(T_{\ell_i}T_{\tau(\ell_i)}^{-1}) = 0$.
    \item[-] If $\ell_i$ is a non-separating curve that does not intersect $\tau(\ell_i)$, then we may use Formula  \eqref{eq:johnsonBPM} to compute $J([T_{\ell_i},\tau])$.
    \item[-] If $\ell_i$ is a non-separating curve that intersects $\tau(\ell_i)$, we may find a sequence of loops $\ell_i = \gamma_0,\gamma_1,\ldots,\gamma_k = \tau(\ell_i)$ so that $\gamma_j$ does not intersect $\gamma_{j+1}$, and compute each $J(T_{\gamma_j}T_{\gamma_{j+1}}^{-1})$ using Formula \eqref{eq:johnsonBPM}. Such a sequence must exist for $g\geq 3$ by \cite[Theorem~1.9]{Putman}. 
\end{itemize}
With an explicit formula for $v_{\Gamma}$ in hand, let us show how to determine if it represents the trivial element of  $\overline{B(\delta_{\Gamma})}$. Recall that $\overline{B(\delta_{\Gamma})}$ is the cokernel of the map 
\begin{equation*}
    \delta_{\Gamma}-I:\gr_{1}^F(\LH) \ra \gr_{2}^F(\LH)
\end{equation*}
Because the vertex weights of $\Gamma$ are all 0, the polarization $Q_{\Gamma}$, viewed as a map $H/Y \to Y$, is nonsingular, and  the map $\delta_{\Gamma}-I$ above is invertible after tensoring with $\Q$. An explicit inverse is 
\begin{equation}
\label{eq:deltaMinusIInverse}
(\delta_{\Gamma}-I)_{\Q}^{-1}(y\wedge y' \wedge h) = \tfrac{1}{2}\left(Q_{\Gamma}^{-1}y \wedge y' \wedge h + y\wedge Q_{\Gamma}^{-1}y'  \wedge h - Q_{\Gamma}^{-1}y\wedge Q_{\Gamma}^{-1}y' \wedge Q_{\Gamma} h  \right) 
\end{equation}  
for $y,y'\in Y$ and $h\in H$. Thus, the class $v_{\Gamma}$ is trivial if and only if $u_{\Gamma} = (\delta_{\Gamma}-I)_{\Q}^{-1}(v_{\Gamma})$ is integral, i.e., lies in $\gr_1^F(L/H)$. Thus, determining triviality of $v_{\Gamma}$ is almost as simple as computing the coordinates of $u_{\Gamma}$ in $\gr_2^F(L)_{\Q}$ using a basis of $H$ and seeing if they are integral, yet we still need to account for taking the quotient by $H$.  Recall from Equation \eqref{eq:omegaV} that $\omega = \alpha_1\wedge\beta_1 + \cdots + \alpha_g\wedge \beta_g$ for any symplectic basis $\alpha_1,\ldots,\alpha_g,\beta_1,\ldots,\beta_g$ of $H$. Therefore, any two representatives in $L$ of an element in $\LH$ differ only in coordinates of the form $\aab{i}{j}{i}$ or $\abb{i}{i}{j}$. In conclusion, we have the following way to determine nontriviality of the Ceresa class.

\begin{proposition}
\label{prop:ceresaNontrivial}
Suppose that the vertex weights of $\Gamma$ are all $0$. The Ceresa class $\nu(T_{\Gamma})$ is nontrivial if  $u_{\Gamma}$ has a coordinate of the form $\alpha_i\wedge\alpha_j\wedge \beta_k$, where $i,j,k$ are distinct, whose coefficient is not integral.
\end{proposition}

We will now illustrate this analysis in some examples.


\begin{example}
\label{ex:K4}
Suppose $\Gamma$ is a tropical curve whose underlying vertex-weighted  graph is $K_4$ (the left graph in  Figure \ref{fig:graphs}). Then $\Gamma$ is Ceresa nontrivial. 
\end{example}

By Proposition \ref{prop:ceresaRealEdgelengths}, it suffices to show that $\nu(T_{\Gamma}) \neq 0$ whenever $\Gamma$ has integral edge lengths.   Let $\ell_1,\ldots,\ell_6$ be configuration of essential closed curves illustrated in Figure  \ref{fig:K4surface}. Because the vertex weights of $\Gamma$ are all 0, this arrangement of curves is Lagrangian, and hence $\nu(T_{\Gamma}) \in \overline{A(\delta_{\Gamma})}$. Let $v_{\Gamma}$ denote the projection  of $\nu(T_{\Gamma})$ to $\overline{B(\delta_{\Gamma})}$. 
Let $\tau$ be the hyperelliptic involution given by a rotation of $180^{\circ}$ through the axis horizontal to the page. Observe that  $[T_{\ell_i}, \tau] = 1$ for $i=1,3,4,6$. The remaining two commutators $[T_{\ell_2},\tau] = T_{\ell_2} T_{\tau(\ell_2)}^{-1}$ and $[T_{\ell_5},\tau] = T_{\ell_5} T_{\tau(\ell_5)}^{-1}$ are bounding pair maps, so we may compute their images under $J$ using Equation \eqref{eq:johnsonBPM}. With respect to the basis of $H=H_1(\Sigma_3,\Z)$ in Figure \ref{fig:basisJohnsonHom}, 
these are
\begin{align*}
&J([T_{\ell_2},\tau]) =  \abb{1}{1}{2} \\
&J([T_{\ell_5},\tau]) =  -\abb{2}{1}{2} - \abb{2}{2}{3} + \abb{2}{1}{3}.
\end{align*}
By Equation \ref{eq:nuBdelta}, the class $v_{\Gamma}$ is
\begin{equation}
\label{eqn:muK4}
v_{\Gamma} = c_2 \cdot \abb{1}{1}{2} + c_5\cdot (-\abb{2}{1}{2} - \abb{2}{2}{3} + \abb{2}{1}{3}).
\end{equation}
Next we compute  $u_{\Gamma} = (\delta_{\Gamma}-I)_{\Q}^{-1}(v_{\Gamma})$
using Equation \ref{eq:deltaMinusIInverse}:
\begin{align*}
    \det(Q_{\Gamma})&\, u_{\Gamma} = \\
    &-(c_1c_2c_3 + c_1c_2c_4 + c_1c_2c_5 + c_2c_3c_5 + c_2c_4c_5 + c_2c_3c_6 + c_2c_4c_6 + c_2c_5c_6 \\
    &\hspace{10pt} + c_3c_5c_6) \, \cdot \aab{1}{2}{1} \\
    & -c_2c_4(c_1 + c_5 + c_6) \, \cdot \aab{1}{3}{1} + c_1c_5c_6\, \cdot \aab{2}{3}{1} \\
    &+ (c_2c_3c_5 + c_2c_4c_5 + c_3c_4c_5 + c_2c_3c_6 + c_2c_4c_6 + c_2c_5c_6 + c_3c_5c_6)\, \cdot \aab{1}{2}{2} \\
    &+c_2c_4c_6\, \cdot \aab{1}{3}{2} +c_1c_5(c_2 + c_4 + c_6)\, \cdot \aab{2}{3}{2} \\
    &- c_3c_4c_5\, \cdot \aab{1}{2}{3} + c_2c_4c_5\, \cdot \aab{1}{3}{1} - c_1c_4c_5\,\cdot \aab{2}{3}{3},
\end{align*}
where
\begin{equation}
\label{eq:QK4}
    Q_{\Gamma} = \begin{pmatrix}
    c_1 + c_5 + c_6 & -c_6 & -c_5 \\
    -c_6 & c_2+c_4+c_6 & -c_4 \\
    -c_5 & -c_4 & c_3+c_4+c_5
    \end{pmatrix}.
\end{equation}
By Equation \eqref{eq:SymanzikPolynomial}, $\det(Q_{\Gamma})$ is the first Symanzik polynomial of $G$. Observe that the absolute value of the coordinates of $\det(Q_{\Gamma})u_{\Gamma}$ consist of a sum of monomials of the form $c_T$ for spanning trees $T$, each appearing with coefficient  $1$. Thus 
for any positive value of the $c_i$'s, each coordinate of $u_{\Gamma}$ has absolute value strictly between $0$ and $1$. By Proposition \ref{prop:ceresaNontrivial}, the Ceresa class $\nu(T_{\Gamma})$ is nontrivial. 

\begin{figure}[tbh]
	\centering
	\includegraphics[height=25mm]{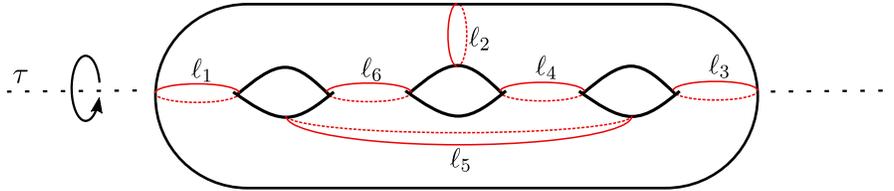}
	\caption{An arrangement of simple closed curves on $\Sigma_3$ dual to $K_4$} \label{fig:K4surface}
\end{figure}

\begin{remark} \label{rem:zharkov}
This is an opportune moment to remark on the relation between the definition of tropical Ceresa class in the present paper and the definition given by Zharkov in \cite{zharkov}.  Consider the element $w_\Gamma = (\delta_\Gamma-1) v_\Gamma$, which lies in $\gr_3^F(L/H)$.  If $v_\Gamma$ lies in $(\delta_\Gamma-1) \gr_1^F(L/H)$, then certainly $w_\Gamma$ lies in $(\delta_\Gamma-1)^2 \gr_1^F(L/H)$.  So the Ceresa class maps to 
\begin{equation*}
\gr_3^F(L/H) / (\delta_\Gamma-1)^2 \gr_1^F(L/H).
\end{equation*}

\noindent Using the expression for $v_\Gamma$ in Equation \eqref{eqn:muK4}, we see that

\begin{equation*}
w_\Gamma = -2c_2c_5 \cdot \bbb{1}{2}{3}. 
\end{equation*}

The group  $(\delta_\Gamma-I)^2\gr_1^F(L/H)$ is generated by $(\delta_\Gamma-I)^2(\aab{i}{j}{k})$ for all $i,j,k \in \{1,2,3\}$ with $i<j$. Because
\begin{equation*} 
(\delta_\Gamma-I)^2(\aab{i}{j}{k}) = 2\,Q\alpha_i \wedge Q\alpha_j \wedge \beta_k = 2\left|\begin{array}{cc}q_{si} & q_{sj}\\ q_{ti} & q_{tj} \end{array}\right| \bbb{1}{2}{3}
\end{equation*}
where $(s, t, k)$ is an even permutation of $(1,2,3)$, 
we see that $(\delta_\Gamma-I)^2\gr^F_1 (L/H)$ is generated by 2 times the $2\times 2$ minors of the symmetric matrix $Q_{\Gamma}$ from Equation \eqref{eq:QK4}. Thus, this subgroup is generated by
\begin{equation*}
\begin{array}{cc}
2(c_1c_4-c_2c_5), & 2(c_1c_4-c_3c_6),     \\
2(c_2c_5+c_4c_5+c_4c_6+c_5c_6), & 2(c_2c_5+c_1c_2+c_1c_6+c_2c_6), \\
2(c_2c_5+c_1c_3+c_1c_5+c_3c_5), & 2(c_2c_5+c_2c_3+c_2c_4+c_3c_4).
\end{array}
\end{equation*}
So the Ceresa class of $T_\Gamma$ is nontrivial whenever $-2c_2c_5$ does not lie in the subgroup of $\Z$ generated by the six integers above.  This is precisely the condition Zharkov computes in \cite[\S~3.2]{zharkov} for the algebraic nontriviality of his Ceresa cycle for a tropical curve with underlying graph $K_4$.  It remains to be understood whether this relation between our tropical Ceresa class and Zharkov's holds in general.

\end{remark}

\begin{remark}
We observe that the element $u_{\Gamma} = (\delta_{\Gamma}-I)_{\Q}^{-1}(v_{\Gamma})$ is a point in a $6$-dimensional torus (identified by our choice of basis here with $(\R/\Z)^6$) which is zero if and only if the Ceresa class vanishes.  What's more, $u_\Gamma$ does not change when the edge lengths are scaled simultaneously, since both $\delta_{\Gamma}-1$ and $v_\Gamma$ are homogeneous of degree $1$ in the edge lengths.  The space $M(K_4)$ of tropical curves with underlying graph $K_4$ is a positive orthant in $\R^6$, or more precisely the quotient of this orthant by the action of the automorphism group $S_4$ (see \cite{chanlectures} for a full description) and the class $u_\Gamma$ can be thought of as a map from the projectivization of this orthant to the $6$-torus.  The content of Example~\ref{ex:K4} is then that the image of this map does not include $0$.  It would be interesting to understand whether this map can be naturally extended to the whole tropical moduli space of genus $3$ curves.
\end{remark}

\begin{remark}
\label{rmk:K4All1}
Consider the case $c_1 = \cdots = c_6 = 1$. The invariant factors of $Q_{\Gamma}$ are $q_1=1$, $q_2=4$, and $q_3=4$, and hence the projection $\overline{A(\delta_{\Gamma})} \to \overline{B(\delta_{\Gamma})}$ is an isomorphism. We compute $v_{\Gamma}$ and $u_{\Gamma}$ to be
\begin{align*}
    v_{\Gamma} =& \abb{1}{1}{2} -\abb{2}{1}{2} - \abb{2}{2}{3} + \abb{2}{1}{3},\\
    u_{\Gamma} =& \frac{1}{16}(-9\, \aab{1}{2}{1} -3 \,  \aab{1}{3}{1} + \, \aab{2}{3}{1} + 7\, \cdot \aab{1}{2}{2} +\, \aab{1}{3}{2} \\
    &+ 3\,  \aab{2}{3}{2} - \, \aab{1}{2}{3} + \, \aab{1}{3}{1} - \, \aab{2}{3}{3}).
\end{align*}
From this, we see that $v_{\Gamma}$ has order 16 in $\overline{B(\delta_{\Gamma})}$, and therefore  the Ceresa class $\nu(T_{\Gamma})$ also has order 16 in $\overline{A(\delta_{\Gamma})}$.  
\end{remark}

There were two key features in this example: each $[T_{\ell_i},\tau]$ was either trivial or a bounding pair map, and each coordinate of $(\delta_{\Gamma}-I)_{\Q}^{-1}(\nu(T_{\Gamma}))$ had absolute value strictly between 0 and 1. In the next example, neither of these properties will hold.

\begin{example}
\label{ex:TL3}
Suppose $\Gamma$ is a tropical curve whose underlying vertex-weighted  graph is $TL_3$ (the middle graph in  Figure \ref{fig:graphs}). Then $\Gamma$ is Ceresa nontrivial. 
\end{example}
As in the previous example, it suffices to show that $\nu(T_{\Gamma})$ is nontrivial whenever $\Gamma$ has integral edge lengths. Let $v_{\Gamma}$ denote the image of $\nu(T_{\Gamma})$ in $\overline{B(\delta_{\Gamma})}$.  Let $\ell_1,\ldots,\ell_9$ be the configuration of essential closed curves in Figure \ref{fig:L3surface}, and choose  $\tau$ to be rotation by $180^{\circ}$ through the axis horizontal to the page.  
To compute each $J([T_{\ell_i},\tau])$, we will use the symplectic basis of $H=H_1(\Sigma_4,\Z)$ displayed on the right in Figure \ref{fig:L3surface}; this yields a much nicer expression for $u_{\Gamma} = (\delta_{\Gamma}-I)_{\Q}^{-1}(v_{\Gamma})$.     Clearly, $[T_{\ell_i},\tau] = 1$ for $i=1,2,3,4$.  Next, $[T_{\ell_i},\tau] = T_{\ell_i}T_{\tau(\ell_i)}^{-1}$ are bounding pair maps when $i=5,6,7,9$, hence $J([T_{\ell_i},\tau])$ may be computed using Formula \eqref{eq:johnsonBPM}. The only remaining loop is $\ell_8$, which intersects $\tau(\ell_8)$. However, $T_{\ell_8}T_{\tau(\ell_8)}^{-1} = (T_{\ell_8}T_{\ell_9}^{-1})(T_{\ell_9}T_{\tau(\ell_8)}^{-1})$ expresses $T_{\ell_8}T_{\tau(\ell_8)}^{-1}$ as a product of bounding pair maps, which we may use to compute $J([T_{\ell_8},\tau])$. Then
\begin{align*}
v_{\Gamma}= &-c_6 \cdot \abb{2}{1}{2} + -c_5\cdot \abb{1}{1}{2} \\
&+ (-c_5-c_7+c_8-c_9)( \abb{1}{1}{3} + \abb{1}{1}{4}) \\
&+(c_6-c_7+c_8+c_9)(\abb{2}{2}{3} + \abb{2}{2}{4}).
\end{align*}
Next we compute  $u_{\Gamma}$
using Equation \ref{eq:deltaMinusIInverse}:
{\small
\begin{align*}
&\det(Q_{\Gamma})\, u_{\Gamma}  = \\
&c_5(c_1c_3c_4 + c_1c_3c_6 + c_1c_4c_6 + c_3c_4c_6 + 2c_1c_3c_8 + 2c_1c_4c_8 + 2c_3c_6c_8 + 2c_4c_6c_8) \cdot \aab{1}{2}{1} \\
&c_4(c_1 + c_6)(c_2c_5 - c_2c_8 - c_5c_8 + c_2c_9 + c_5c_9 + c_2c_7 + c_5c_7) \cdot \aab{1}{3}{1} \\
&c_3(c_1 + c_6)(c_2c_5 - c_2c_8 - c_5c_8 + c_2c_9 + c_5c_9 + c_2c_7 + c_5c_7) \cdot \aab{1}{4}{1} \\
&c_6(c_2c_3c_4 + c_2c_3c_5 + c_2c_4c_5 + c_3c_4c_5 + 2c_2c_3c_7 + 2c_2c_4c_7 + 2c_3c_5c_7 + 2c_4c_5c_7) \cdot \aab{1}{2}{2} \\
&-c_4(c_2 + c_5)(c_1c_6 + c_1c_8 + c_6c_8 + c_1c_9 + c_6c_9 - c_1c_7 - c_6c_7) \cdot \aab{2}{3}{2} \\
&-c_3(c_2 + c_5)(c_1c_6 + c_1c_8 + c_6c_8 + c_1c_9 + c_6c_9 - c_1c_7 - c_6c_7) \cdot \aab{2}{4}{2} \\
&c_5c_6(c_3c_4 + c_3c_8 + c_4c_8 - c_3c_9 - c_4c_9 + c_3c_7 + c_4c_7) \cdot \aab{1}{2}{3} \\
&c_4c_6(c_2c_5 - c_2c_8 - c_5c_8 + c_2c_9 + c_5c_9 + c_2c_7 + c_5c_7) \cdot \aab{1}{3}{3} \\
&c_3c_6(c_2c_5 - c_2c_8 - c_5c_8 + c_2c_9 + c_5c_9 + c_2c_7 + c_5c_7) \cdot \aab{1}{4}{3} \\
&-c_4c_5(c_1c_6 + c_1c_8 + c_6c_8 + c_1c_9 + c_6c_9 - c_1c_7 - c_6c_7) \cdot \aab{2}{3}{3} \\
&-c_3c_5(c_1c_6 + c_1c_8 + c_6c_8 + c_1c_9 + c_6c_9 - c_1c_7 - c_6c_7) \cdot \aab{2}{4}{3} \\
&c_5c_6(c_3c_4 + c_3c_8 + c_4c_8 - c_3c_9 - c_4c_9 + c_3c_7 + c_4c_7) \cdot \aab{1}{2}{4} \\
&c_4c_6(c_2c_5 - c_2c_8 - c_5c_8 + c_2c_9 + c_5c_9 + c_2c_7 + c_5c_7) \cdot \aab{1}{3}{4} \\
&c_3c_6(c_2c_5 - c_2c_8 - c_5c_8 + c_2c_9 + c_5c_9 + c_2c_7 + c_5c_7) \cdot \aab{1}{4}{4} \\
&-c_4c_5(c_1c_6 + c_1c_8 + c_6c_8 + c_1c_9 + c_6c_9 - c_1c_7 - c_6c_7) \cdot \aab{2}{3}{4} \\
&-c_3c_5(c_1c_6 + c_1c_8 + c_6c_8 + c_1c_9 + c_6c_9 - c_1c_7 - c_6c_7) \cdot \aab{2}{4}{4},
\end{align*}
}
where 
\begin{equation*}
    Q_{\Gamma} = \begin{pmatrix}
    c_1 +  c_6 & 0 & c_6 & c_6 \\
    0 & c_2+c_5 & c_5 & c_5 \\
    c_6 & c_5 & c_3+c_5+c_6+c_7+c_8+c_9 & c_5+c_6+c_7+c_8+c_9 \\
    c_6 & c_5 & c_5+c_6+c_7+c_8+c_9 & c_4+c_5+c_6+c_7+c_8+c_9
    \end{pmatrix}.
\end{equation*}
By Equation \eqref{eq:SymanzikPolynomial}, $\det(Q_{\Gamma})$ is the first Symanzik polynomial of $G$. The coordinates of $\det(Q_{\Gamma})u_{\Gamma}$ of the form $\aab{i}{j}{k}$, for $i,j,k$ distinct, are a sum of monomials $c_T$ for spanning trees $T$, each appearing with coefficient  $\pm 1$. 
Thus for any positive value of the $c_i$'s, each coordinate of these coordinates is strictly between $-1$ and $1$. If they are all equal to $0$, then
\begin{align*}
(c_1 + c_6)c_7 = c_1c_6 + c_1c_8 + c_6c_8 + c_1c_9 + c_6c_9, \\
(c_2 + c_5)c_8 = c_2c_5 + c_2c_9 + c_5c_9 + c_2c_7 + c_5c_7, \\
(c_3 + c_4)c_9 = c_3c_4 + c_3c_8 + c_4c_8 + c_3c_7 + c_4c_7. 
\end{align*}
Solving for $c_7$ in the first equation and substituting this expression in the second equation yields
\begin{equation*}
c_1c_2c_5 + c_1c_2c_6 + c_1c_5c_6 + c_2c_5c_6 + 2c_1c_2c_9 + 2c_1c_5c_9 + 2c_2c_6c_9 + 2c_5c_6c_9 = 0
\end{equation*}
which cannot happen if every $c_i$ is positive. Therefore, the Ceresa class $\nu(T_{\Gamma})$ is nontrivial by Proposition \ref{prop:ceresaNontrivial}.

\begin{figure}
\centering
\begin{minipage}{.5\textwidth}
  \centering
  \includegraphics[width=0.9\linewidth]{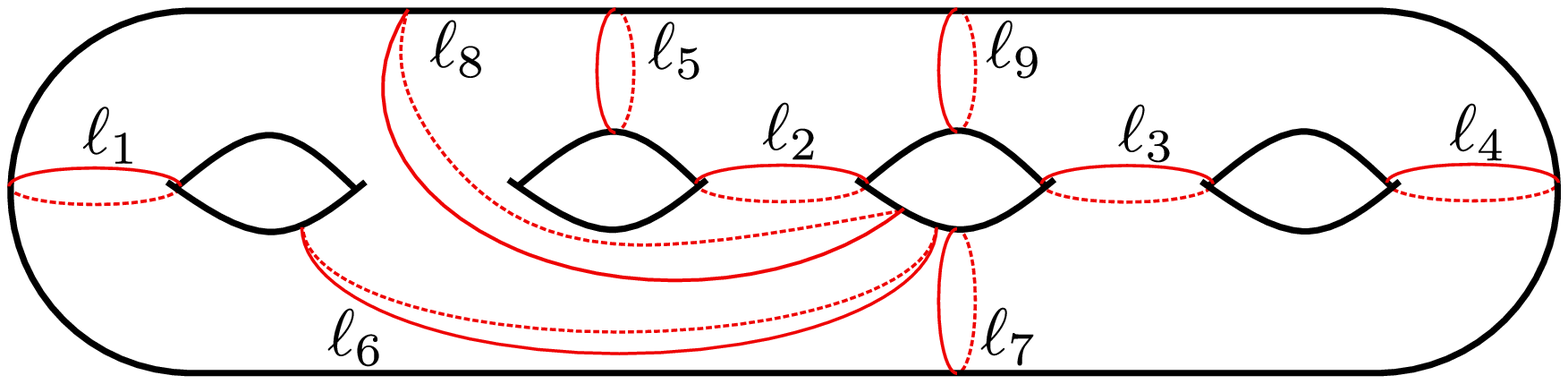}
\end{minipage}%
\begin{minipage}{.5\textwidth}
  \centering
  \includegraphics[width=0.9\linewidth]{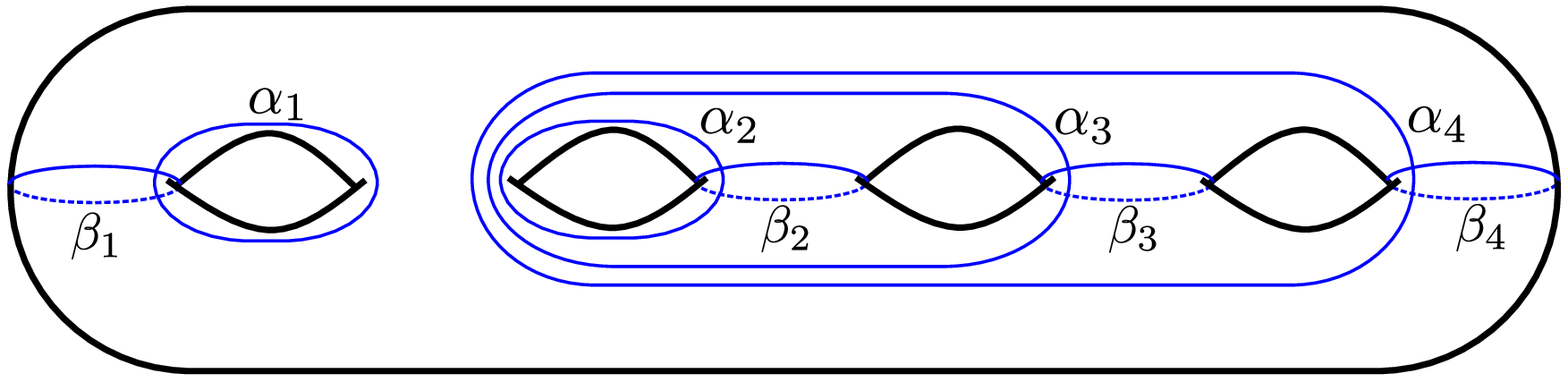}
\end{minipage}
\caption{Left: an arrangement of simple closed curves on $\Sigma_4$ dual to $TL_3$. Right: the basis we use to compute $\nu(T_{\Gamma})$, oriented so that $\hat{i}(\alpha_i,\beta_i) = 1$} \label{fig:L3surface}
\end{figure}

In the previous two examples, each collection of curves is Lagrangian, and hence $\nu(T_{\Gamma})$ lies in $\overline{A(\delta_{\Gamma})}$. However, the Ceresa class $\nu(T_{\Gamma})$ may not live in $\overline{A(\delta_{\Gamma})}$, as we shall see in the following example. 

\begin{example}
\label{example:thetaGraph}
Let $\Gamma$ be a tropical curve whose underlying graph is a theta graph. Suppose the two vertices each have weight 1, and each edge has length 1, see the right graph in Figure \ref{fig:graphs}.
Then $\nu(T_{\Gamma})\notin \overline{A(\delta_{\Gamma})}$, but $3 \, \nu(T_{\Gamma}) \in \overline{A(\delta_{\Gamma})}$. In particular $\Gamma$ is Ceresa nontrivial.
\end{example}

Let $\ell_1,\ell_2,\ell_3$ be the  configuration of essential closed curves illustrated in Figure \ref{fig:ExampleNotInF2}.
Consider the following  basis of $H_1(\Sigma_4,\Z)$, written in terms of the basis in Figure \ref{fig:basisJohnsonHom}: 
\begin{equation*}
\alpha_2' = -\alpha_2,\; \alpha_3' = 2\alpha_2 + \alpha_3,\; \beta_2' = \beta_3-2\beta_2,\; \beta_3' = \beta_2
\end{equation*}
and  $\alpha_i' = \alpha_i$, $\beta_i' = \beta_i$ for $i=1,4$. On this basis, $(\delta_{\Gamma}-I)(\alpha_2') = \beta_2'$, 
$(\delta_{\Gamma}-I)(\alpha_3') = 3\beta_3'$, and $(\delta_{\Gamma}-I)(\alpha_i')$, $(\delta_{\Gamma}-I)(\beta_i')=0$ otherwise.  So  $Y' = \Span_{\Z}\{\beta_2', 3\beta_3' \}$, and $Y = \Span_{\Z}\{\beta_2', \beta_3' \}$. Next, we compute

\begin{align*}
J([T_{\Gamma},\tau]) &= J([T_{\ell_1},\tau]) + T_{\ell_1} \cdot  J([T_{\ell_2},\tau]) + T_{\ell_1}T_{\ell_2}\cdot J([T_{\ell_3},\tau])\\
&=\alpha_1'\wedge \beta_1' \wedge  \beta_2' + \alpha_1'\wedge \beta_1' \wedge \beta_3' +  \alpha_2' \wedge \beta_2'\wedge \beta_3'.
\end{align*}
Clearly, $\alpha_2' \wedge \beta_2'\wedge \beta_3' \in F_2L$,  and $\alpha_1'\wedge \beta_1' \wedge \beta_2' \in (\delta_{\Gamma}-I)(L)$ by Proposition \ref{prop:F1KernelDeltaITrivial}. However, the simple wedge $\alpha_1'\wedge \beta_1' \wedge \beta_3' $ is not contained in $F_2L + (\delta_{\Gamma}-I)L+H$, so $\nu(T_{\Gamma})\notin \overline{A(\delta_{\Gamma})}$. Nevertheless, the class $3\,\nu(T_{\Gamma})$ lies in $\overline{A(\delta_{\Gamma})}$ because $Y'\leq 3\, Y$.

\begin{figure}[tbh]
	\centering
	\includegraphics[height=30mm]{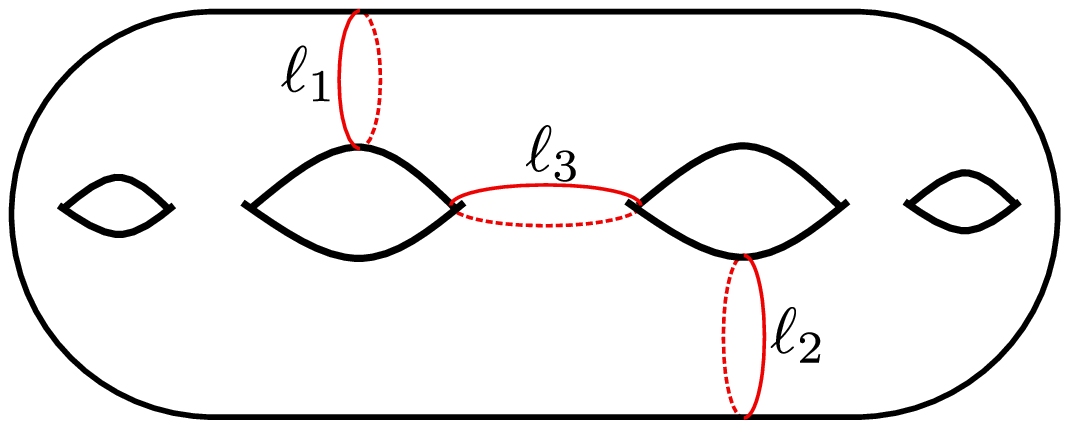} 
	\caption{An arrangement of simple closed curves on $\Sigma_4$ dual to a theta graph with a weight on each vertex} \label{fig:ExampleNotInF2}
\end{figure}

\begin{remark}
The $\ell$-adic Ceresa class of a hyperelliptic algebraic  curve is trivial, which we may interpret as a property of hyperelliptic Jacobians via the Torelli theorem. Nevertheless, the tropical analog of this does not hold because the property of being hyperelliptic cannot be determined by the Jacobian alone. A tropical curve whose Jacobian is isomorphic to the Jacobian of a hyperelliptic tropical curve, as polarized tropical abelian varieties, is said to be of \textit{hyperelliptic type}. By \cite[Theorem~1.1]{Corey}, the tropical curve from Example \ref{example:thetaGraph} is of hyperelliptic type, yet it is Ceresa nontrivial. 
Thus, the Ceresa class for a tropical curve is not an invariant of its Jacobian and can be used to distinguish hyperelliptic tropical curve from tropical curves of hyperelliptic type. One can ask whether the tropical Ceresa class is trivial exactly for hyperelliptic tropical curves. Translating this question into topological terms is a bit subtle, because the multitwist $T_\Gamma$ associated to a hyperelliptic tropical curve $\Gamma$ is not necessarily hyperelliptic; that is, there may not be a hyperelliptic involution in the mapping class group which commutes with $T_\Gamma$.  This is related to the issue that there exist hyperelliptic tropical curves which are not tropicalizations of any degenerating algebraic hyperelliptic curve.  Consider, for instance, the curve in \ref{fig:3balloon}.

\begin{figure}[tbh!]
    \centering
    \includegraphics[width=3cm]{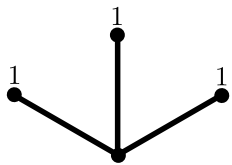}
    \caption{A hyperelliptic tropical curve that is not the tropicalization of a hyperelliptic curve}
    \label{fig:3balloon}
\end{figure}

This is a tropical hyperelliptic curve which is known not to be the tropicalization of a hyperelliptic curve.  The genus-3 mapping class $T_{\Gamma}$ corresponding to this curve is a product of three commuting \textit{separating} Dehn twists; this mapping class is not hyperelliptic, but $J([T_{\Gamma},\tau])$ vanishes for {\em all} $\tau$ since $T_{\Gamma}$ itself lies in the kernel of the Johnson homomorphism.  Indeed, it is easy to describe a hyperelliptic quasi-involution which commutes with $T_{\Gamma}$ not only up to the Johnson kernel, but on the nose: the product of the three Dehn {\em half}-twists about the separating curves.

Nonetheless, the explicit criterion given in \cite[Theorem 4.13]{AminiBakerBrugalleRabinoff} shows that a hyperelliptic tropical curve yields a hyperelliptic mapping class under mild conditions; for instance, it is enough that the underlying graph of $\Gamma$ be $2$-vertex-connected.  So one might ask:  if $\gamma$ is a positive multitwist whose Ceresa class vanishes and whose corresponding graph is $2$-vertex-connected, is $\gamma$ hyperelliptic?
\end{remark}

\subsection*{Acknowledgments}  The first author is partially supported by NSF-RTG grant 1502553 and  ``Symbolic Tools in Mathematics and their Application'' (TRR 195, project-ID 286237555).  The second author is partially supported by NSF-DMS grant 1700884.  The third author is partially supported from the Simons Collaboration on Arithmetic Geometry, Number Theory and Computation.  We are grateful for helpful comments and suggestions from Matt Baker, Benson Farb, Daniel Litt, Andrew Putman, and Bjorn Poonen.

\subsection*{Competing interests}   The authors declare none.

\bibliographystyle{amsplain}
\bibliography{CELbib}

\end{document}